\documentclass[12pt]{article}
\usepackage{amsmath,amssymb}
\usepackage[mathscr]{euscript}
\usepackage{graphicx}
\usepackage[all]{xypic}

\newtheorem{rema}{Remark}[section]
\newtheorem{defi}{Definition}[section]
\newtheorem{lemm}{Lemma}[section]
\newtheorem{theo}{Theorem}[section]
\newtheorem{coro}{Corollary}[section]

\newcommand{\bbox}{\normalsize {}%
        \nolinebreak \hfill $\blacksquare$ \medbreak \par}

\newcommand{\R}{\mathbb{R}}
\newcommand{\C}{\mathbb{C}}

\newcommand{\N}{\mathbb{N}}
\newcommand{\U}{\mathbb{U}}
\newcommand{\V}{\mathbb{V}}
\newcommand{\T}{\mathbb{T}}
\newcommand{\F}{\mathbb{F}}
\newcommand{\D}{\mathbb{D}}

\newcommand{\gd}{\overleftrightarrow{\partial}}
\newcommand{\sharpgdt}{\overleftrightarrow{\sharp_t}}
\newcommand{\tx}{\triangleright}
\newcommand{\td}{\triangleleft}
\newcommand{\lc}{\boldsymbol{[}\!\!\boldsymbol{[}}
\newcommand{\rf}{\boldsymbol{]}\!\!\boldsymbol{]}}

\usepackage{stmaryrd}                       
\newcommand\lent{\llbracket }               
\newcommand\rent{\rrbracket }               
\newcommand\chapo[1]{\widehat{#1} }         
\newcommand\source{{\phi}^\triangleright}   
\newcommand\but{\phi^\triangleleft}         
\newcommand\refeq[1]{(\ref{#1})}            
\newcommand\intit{>\!\!\!\!\!\!\!\int}      
\newenvironment{proof}{\noindent\emph{Proof} ---}{\bbox\vspace{0,15cm}}

\title{First integrals for non linear hyperbolic equations}
\author{Dikanaina \textsc{Harrivel}\footnote{Institut Camille Jordan, UMR CNRS 5208, Universit{\'e} Claude Bernard Lyon 1, 13 boulevard du 11 novembre 1918, 69622 Villeurbanne cedex, France, \textsf{harrivel@math.univ-lyon1.fr}}, Fr{\'e}d{\'e}ric \textsc{H{\'e}lein}\footnote{Institut de Math{\'e}matiques de Jussieu, UMR CNRS 7586 Universit{\'e} Denis Diderot Paris 7,
Case 7012, 2 place Jussieu
75251 Paris Cedex 5, France, \textsf{helein@math.jussieu.fr}}
}

\begin{document}
\maketitle
\begin{abstract}
Given a solution of a nonlinear wave equation on the flat space-time (with a real analytic nonlinearity), we relate its Cauchy data at two different times by nonlinear representation formulas in terms of asymptotic series. We first show how to construct formally these series by mean of generating functions based on an algebraic framework inspired by the construction of Fock spaces in quantum field theory. Then we build an analytic setting in which all these constructions really make sense and give rise to convergent series. 
\end{abstract}
$\pmb{[}u\pmb{]}$ $\boldsymbol{[}v\boldsymbol{]}$ $[v]$
It is well-known that, for hyperbolic wave equations like, for instance, the linear Klein--Gordon equation on the space-time $\R\times \R^n$:
\begin{equation}\label{kg}
{\partial ^2u\over \partial t^2} - \Delta u + m^2 u = 0,
\end{equation}
one can construct first integrals, i.e. conserved quantities. One of the most important example is the energy $E_t[u]:= {1\over 2}\int_{\R^n}\left( ({\partial u\over \partial t})^2 + |\vec{\nabla} u|^2 + m^2u^2\right)|_{x^0=t}\ d\vec{x}$, which is particularly useful in the analysis of the solutions of (\ref{kg}). (Here we denote by $x = (x^0,\vec{x})\in \R\times \R^n$ a point in space-time, ${\partial u\over\partial t}:= {\partial u\over\partial x^0}$ and $\vec{\nabla}u := \left( {\partial u\over \partial x^1}, \cdots , {\partial u\over \partial x^n}\right)$.) By claiming that the family of functionals $\left(E_t\right)_{t\in \R}$ is a first integral we mean that, for any given solution $u$ of (\ref{kg}), the value of $E_t[u]$ does not depend on $t$. Equation (\ref{kg}) possesses however other conserved quantities such as
\begin{equation}\label{Iphi}
I^\varphi_t[u]:= \int_{\R^n}\left(u{\partial \varphi\over \partial t} - {\partial u\over \partial t}\varphi \right)|_{x^0=t}\ d\vec{x},
\end{equation}
where $\varphi$ is a fixed solution of (\ref{kg}). One interpretation of these functionals is based on Noether's theorem: solutions to (\ref{kg}) are the critical points of the functional $\mathcal{L}_0[u]:= \int_{\R\times \R^n}{1\over 2}\left( ({\partial u\over \partial t})^2 - |\vec{\nabla} u|^2 - m^2u^2\right)dx$ and hence a first integral is associated to each symmetry. The conservation of the energy $E_t$ is then a consequence of the invariance of this problem by translations in time, whereas the conservation of functionals $I^\varphi_t$ is due to the fact that the Lagrangian functional $\mathcal{L}$ is invariant up to a boundary term by the substitution $u\longmapsto u + s\varphi$. Another point of view, which is closer to differential Geometry, is to consider the set $\mathcal{E}_0$ of all solutions to (\ref{kg}) as a manifold (here we stay vague about the choice of the topology). Then to each time $t$ we associate a system of coordinates on $\mathcal{E}_0$ which is nothing but the Cauchy data $[u]_t:= \left(u(t,\cdot),{\partial u\over \partial t}(t,\cdot)\right)$ and the conservation of, say $I^\varphi_t$, means that a functional $I^\varphi$ can be consistently constructed on $\mathcal{E}_0$ by using, for each time $t$, the expression $I^\varphi_t$ on the coordinate system $u\longmapsto [u]_t$. Note that the functionals $I^\varphi$ play an important role in the quantization of equation (\ref{kg}), since if $\varphi$ is a (complex) solution of (\ref{kg}) of the form $e^{\pm ik\cdot x}$, then the quantization of $I^\varphi$ leads to creation and annihilation operators (see \cite{harrivelhelein}, the resulting quantum fields are then of course free). In this paper we are interested in finding analogous conserved quantities for a more general, \emph{non linear} Klein--Gordon equation:
\begin{equation}\label{kgnl}
{\partial ^2u\over \partial t^2} - \Delta u + m^2 u + W(u,\nabla u) = 0,
\end{equation}
where $W:\R\times\R^{n+1}\longrightarrow \R$ and $\nabla u:= \left( {\partial u\over \partial x^0}, {\partial u\over \partial x^1}, \cdots , {\partial u\over \partial x^n}\right)$. We denote by $\mathcal{E}_W$ the set of solutions of the \emph{non linear} Klein--Gordon equation (\ref{kgnl}). For instance if $W(y,z_0,z_1,\cdots ,z_n) = V'(y)$, where $V:\R\longrightarrow \R$ is a real analytic function, solutions of (\ref{kgnl}) are the critical points of $\mathcal{L}_V[u]:= \int_{\R\times \R^n}{1\over 2}\left( ({\partial u\over \partial t})^2 - |\vec{\nabla} u|^2 - m^2u^2 - V(u)\right)dx$ and in this case the energy $E_t^V[u]:= {1\over 2}\int_{\R^n}\left( ({\partial u\over \partial t})^2 + |\vec{\nabla} u|^2 + m^2u^2 + V(u) \right)|_{x^0=t}\ d\vec{x}$ is still a conserved quantity. However, as soon as $W$ is not linear, there is no way to find non trivial functions $\varphi$ such that the functionals $I^\varphi_t$ be first integrals of (\ref{kgnl}), as observed in \cite{kijowski,goldstern}. This is the reason why it was proposed in \cite{helein} to look at more general functionals, of the form
\begin{equation}\label{genericseries}
\mathcal{F}_t[u]_t = \sum_{p=1}^\infty \mathcal{F}_{t,p}[u]_t^{\otimes p},
\end{equation}
where $[u]_t$ denotes the Cauchy data at time $t$ and, letting $H$ to be the set of all possible values of Cauchy data at some time, each $\mathcal{F}_{t,p}$ is a linear functional on the $p$-th tensor product $H^{\otimes p}:= H\otimes \cdots \otimes H$ and $[u]_t^{\otimes p}:= [u]_t\otimes \cdots \otimes[u]_t\in H^{\otimes p}$. Hence each functional $u\longmapsto \mathcal{F}_{t,p}[u]_t^{\otimes p}$ is a homogeneous polynomial of degree $p$ and $\mathcal{F}_t$ is a series in $u$. In \cite{harrivel}, the first Author proved that such series can be constructed in the case where $W(u,\nabla u) = u^2$. In the following we recall the content of this paper.

We start with some $\varphi\in\mathcal{E}_0$ (i.e. a solution to the \emph{linear} equation (\ref{kg})) and we use its Cauchy data $[\varphi]_0$ at time 0 to obtain a functional $I^\varphi_0$ on $\mathcal{E}$. Then the first result is that it is possible to build a formal series $\mathcal{F}_t^\varphi$ of the type (\ref{genericseries}) s.t. formally $\mathcal{F}_t^\varphi[u]_t = I^\varphi_0[u]$ for all time $t$. The first term in the series is
\[
\mathcal{F}_{t,1}^\varphi[u]_t:= \int_{\{t\}\times \R^n} d\vec{y} \int_{\{0\}\times \R^n}d\vec{x}\  u(y) {\gd \over \partial y^0} G(y-x){\gd \over \partial x^0} \varphi(x),
\]
where $G$ is the unique tempered distribution on $\R\times \R^n$ which is the solution of ${\partial^2 G\over \partial t^2} - \Delta G + m^2G = 0$ and such that $G|_{x^0=0} = 0$ and ${\partial G\over \partial t}|_{x^0=0} = \delta_0$, the Dirac mass at the origin in $\R^n$. Moreover we have introduced the symbol ${\gd \over \partial y^0}$ denoting\footnote{We adopt here the same notation as in \cite{harrivel} but with an \emph{opposite} sign convention, which agrees with the one used by physicists.}: 
\[
a(y){\gd \over \partial y^0}b(y):= a(y){\partial b\over \partial y^0}(y) - {\partial a\over \partial y^0}(y)b(y),
\]
for any functions $a$ and $b$. However since $\varphi\in \mathcal{E}_0$, $\varphi(y) = \int_{\{0\}\times \R^n}d\vec{x}G(y-x){\gd \over \partial x^0} \varphi(x)$, so that
\[
\mathcal{F}_{t,1}^\varphi[u]_t= \int_{\{t\}\times \R^n}d\vec{y}\  u(y){\gd \over \partial y^0}\varphi(y) = I^\varphi_t[u].
\]
The second term $\mathcal{F}_{t,2}^\varphi[u]_t\otimes [u]_t$ is equal to
\[
- \int\int_{(\{t\}\times \R^n)^2} d\vec{y}_1d\vec{y}_2 \int_{[0,t]\times \R^n}dz \int_{\{0\}\times \R^n}d\vec{x}\  u(y_1)u(y_2) {\gd \over \partial y_1^0}{\gd \over \partial y_2^0} G(y_1-z)G(y_2-z)G(z-x){\gd \over \partial x^0} \varphi(x).
\]
These terms can alternatively be described through diagrams, more precisely trees, by using adapted Feynman rules:
\[
\mathcal{F}_{t,1}^\varphi[u]_t =
\xymatrix{\ar@{-}[r] & y \ar@{-}[r]\ar@{-}[d] & \\
\ar@{-}[r] & x \ar@{-}[r] &},
\quad
\mathcal{F}_{t,2}^\varphi[u]_t^{\otimes 2} = -
\xymatrix{\ar@{-}[r] & y_1 \ar@{-}[r]\ar@{-}[dr] & \ar@{-}[r] & y_2 \ar@{-}[r] &\\
& & z\ar@{-}[ur] && \\
 & \ar@{-}[r] & x\ar@{-}[u]\ar@{-}[r] & & }
\]
The variable $x$ at the bottom horizontal line is associated with the insertion on the right hand side in the integral of ${\gd \over \partial x^0} \varphi(x)$ and with an integration over $\vec{x}$ at time $x^0=0$. Similarly each variable $y$ at the top horizontal line is associated with the insertion on the left hand side in the integral of $u(y) {\gd \over \partial y^0}$ and with an integration over $\vec{y}$ at time $y^0=t$. The intermediate vertex $z$ represents an integration over $[0,t]\times\R^n$ and each intermediate edge is associated to the insertion of $G(\hbox{`up'} - \hbox{`down'})$, where `up' and `down' are the variables at the ends of the corresponding edge.

The second content of the result in \cite{harrivel} (still for $W(u,\nabla u) = u^2$) is that one can choose suitable function spaces for respectively $\varphi$ and $u$ such that the preceding series \emph{converges} and defines a first integral. More precisely we let $H^s(\R^n)$ be the Hilbert space of tempered distributions $v\in \mathcal{S}'(\R^n)$ such that $\epsilon^s\widehat{v}\in L^2(\R^n)$, where $\widehat{v}$ is the Fourier transform of $v$ and $\epsilon(\xi):=\sqrt{m^2+|\xi|^2}$. If $s>n/2$, $\varphi$ is in $\mathcal{C}^0(\R,H^{-s}(\R^n))\cap \mathcal{C}^1(\R,H^{-s-1}(\R^n))$ and $u$ is a solution to (\ref{kgnl}) which belongs to $\mathcal{C}^0([0,T], H^{s+2}(\R^n))\cap \mathcal{C}^1([0,T], H^{s+1}(\R^n))$ for some $T$, then, under a reasonable smallness assumption on $T$ and $||u||$, the functional $\mathcal{F}^\varphi_t$ is defined for all $t\in [0,T]$ and the value of $\mathcal{F}^\varphi_t[u]_t$ does not depend on $t\in [0,T]$.

The proof in \cite{harrivel}, which uses a lot of combinatorics, can be extended without difficulties to nonlinearities of the type $W(u,\nabla u) = u^{p-1}$, for all integer $p\geq 3$, by replacing binary trees by $(p-1)$-nary trees. However it seems difficult to extend it to deal with more general nonlinearities or to systems of hyperbolic wave equations. The purpose of our paper is precisely to present a new principle which allows both to construct and to prove the convergence of functionals of the type (\ref{genericseries}) for more general nonlinearities. It is based on the following heuristic construction.\\

We let $\left(\phi^\tx(x)\right)_{x\in \R\times \R^n}$ and $\left(\phi^\td(x)\right)_{x\in \R\times \R^n}$ be respectively families of `creation' and `annihilation' linear operators, acting on some infinite dimensional vector space $\F$, a classical analogue of Fock spaces used in quantum fields theories. They are formally solutions of the linear Klein--Gordon equation (\ref{kg}), i.e. $\square \phi^\tx+m^2\phi^\tx = \square \phi^\td+m^2\phi^\td = 0$. The space $\F$ contains a particular vector $|0\rangle$ and its dual space $\F^*$ another vector $\langle 0|$, such that $\langle 0|0\rangle = 1$. Operators $\phi^\tx(x)$ and $\phi^\td(x)$ obey to commutation relations $[\phi^\tx(x),\phi^\tx(y)] = [\phi^\td(x),\phi^\td(y)] = 0$ and $[\phi^\td(x),\phi^\tx(y)] = G(x-y)$ and their action on $\F$ and $\F^*$ are such that $\phi^\td(x)|0\rangle = 0$ and $\langle 0|\phi^\tx(x) = 0$. Using these rules, for any $\varphi\in \mathcal{E}_0$ and $u\in \mathcal{E}_W$, one can define formally the following expression
\begin{equation}\label{magic}
\mathcal{F}_t^\varphi[u]_t:= \langle [u]_t|  T\hbox{exp}\left(-\int_{0<z^0<t} W(\phi^\tx(z),\nabla\phi^\tx(z))\phi^\td(z)dz\right) |f_\varphi\rangle,
\end{equation}
where we have set
\begin{equation}\label{magic2}
    \langle [u]_t|:= \langle 0|\hbox{exp}\left(\int_{y^0=t}u(y){\gd\over \partial y^0}\phi^\td(y)d\vec{y} \right)
    \quad \hbox{and}\quad
    |f_\varphi\rangle:= \int_{x^0=0}\phi^\tx(x){\gd\over \partial x^0}\varphi(x)d\vec{x}|0\rangle
\end{equation}
and the symbol $T$ in (\ref{magic}) forces the exponential on his right to be time ordered. Hence if we write
\begin{equation}\label{DS}
\D_s:= \int_{\{s\}\times \R^n}
W(\phi^\tx(z),\nabla\phi^\tx(z))\phi^\td(z)d\vec{z}:=
\int_{\R^n}W(\phi^\tx(s,\vec{z}),\nabla\phi^\tx(s,\vec{z}))\phi^\td(s,\vec{z})d\vec{z},
\end{equation}
then
\[
\begin{array}{r}
\displaystyle T \hbox{exp}\left( -\int_{0<z^0<t}W(\phi^\tx(z),\nabla\phi^\tx(z))\phi^\td(z)dz\right)  =  T\hbox{exp}\left(-\int_0^t\D_sds\right)
\\
\displaystyle  := \sum_{k=0}^\infty (-1)^k \int_{0<s_1<\cdots <s_k<t} \D_{s_k}\cdots \D_{s_1}ds_1\cdots ds_k.
\end{array}
\]
It is clear that, if the expression (\ref{magic}) has some meaning, then it depends only on $[u]_t$ and should be of the form (\ref{genericseries}). We will see that actually, under some hypothesis, it defines a time independant functional.\\

For $s>n/2$ we define $\mathcal{E}_W^{s+1}$ to be the set of pairs $((\underline{t},\overline{t}),u)$, where $\underline{t} < 0 < \overline{t}$ and $u\in \mathcal{C}^0((\underline{t},\overline{t}),H^{s+1}(\R^n))\cap \mathcal{C}^1((\underline{t},\overline{t}),H^s(\R^n))$ is a weak solution of the non linear Klein--Gordon equation (\ref{kgnl}). (Note that the condition $s>n/2$ implies that $u$ is $\mathcal{C}^1$ on $(\underline{t},\overline{t})\times \R^n$.) For all time $t\in (\underline{t},\overline{t})$ we note $||[u]_t||_{s+1}:=  ||u(t,\cdot)||_{H^{s+1}} + ||{\partial u\over \partial t}(t,\cdot)||_{H^s}$. Our main result is
\begin{theo}\label{mainthm}
Assume that $W:\R\times \R^{n+1}\longrightarrow \R$ is an entire real analytic function and let $s>n/2$. Then there exists an entire complex analytic vector field $X = X(z){d\over dz}$ on $\C$, such that $\forall z\in \C$, $X(z) = \sum_{k=0}^\infty a_kz^k$, with $a_k\geq 0$, $\forall k\in \N$, which depends only on $W$ and $s$ such that the following holds.
Let $((\underline{t},\overline{t}),u)\in \mathcal{E}_W^{s+1}$ such that
\[
\sup_{0\leq t<\overline{t}}e^{tX}\left(||[u]_t||_{s+1}\right) < + \infty,
\]
then, for any weak solution $\varphi\in \mathcal{C}^0(\R,H^{-s}(\R^n))\cap \mathcal{C}^1(\R,H^{-s-1}(\R^n))$ of the linear Klein--Gordon equation (\ref{kg}), there exists a family of continuous functionals $\left(\mathcal{F}_t^\varphi\right)_{t\in [0,\overline{t})}$ on $H^{s+1}(\R^n)\times H^s(\R^n)$ of the form (\ref{genericseries}), which satisfies
\begin{equation}\label{conservation}
\forall t\in [0,\overline{t}),\quad
\mathcal{F}_t^\varphi[u]_t = I^\varphi_0[u]_0.
\end{equation}
Moreover $\mathcal{F}_t^\varphi[u]_t$ is given by the expression (\ref{magic}).
\end{theo}
Actually a large part of this paper is devoted to the construction of a framework in which (\ref{magic}) really makes sense. Once this is done Theorem \ref{mainthm} will follow relatively easily. Furthermore we prove a more general result which apply to nonlinear equations of the type $\square u + m^2 u + \mathcal{V}(u,{\partial u\over \partial t}) = 0$, where $s$ is arbitrary and $\mathcal{V}: H^{s+1}(\R^n)\times H^s(\R^n)\longrightarrow H^s(\R^n)$ is a real analytic functional on a neighbourhood of 0 in $H^{s+1}(\R^n)\times H^s(\R^n)$ (if $\mathcal{V}$ is a \emph{local} functional on $u$ and its space-time derivative, i.e. of the form $u\longmapsto W\circ \left(u,\nabla u\right)$, where $W:\R^{n+2}\longrightarrow \R$ is real analytic, \emph{and if $s>n/2$}, this implies Theorem \ref{mainthm}). Note that, as particular examples of applications of Theorem \ref{mainthm}, we can choose $\varphi$ such that $[\varphi]_0 = (0,\delta_{\vec{x}})$, for some $x = (0,\vec{x})$, then it turns out that $\mathcal{F}_t^\varphi[u]_t = u(x)$. Similarly by choosing $[\varphi]_0 = (-\delta_{\vec{x}},0)$ we obtain $\mathcal{F}_t^\varphi[u]_t = {\partial u\over \partial t}(x)$. Furthermore one can develop the expression given in (\ref{magic}) by using a `Wick theorem' (as explained in Section 1 on a simple example) and then recover an expansion where each term is constructed by using adapted Feynman rules from a Feynman diagram modelled on a tree.\\

The paper is organized as follows. In the first section we expound heuristically the basic ideas behind the construction of the formula (\ref{magic}) and we prove formally that it gives us a conserved quantities. We also explain how to recover the series of \cite{harrivel} described previously by a kind of Wick theorem. Roughly speaking our Fock space $\F$ is composed of real analytic functionals on the infinite dimensional space $\mathcal{E}^{s+1}_0$, $|0\rangle$ is the constant functional which takes the value 1 on each $\varphi\in \mathcal{E}^{s+1}_0$, $\langle 0|$ is the `Dirac mass' at the origin in  $\mathcal{E}^{s+1}_0$, $\phi^\tx$ is the multiplication by a linear functional and $\phi^\td$ is a derivation on it. In the second section we introduce the analytic setting and the consistent definitions. The major difficulty is that the operators built from $\phi^\td$ cannot be bounded and hence their exponentials or time ordered exponentials cannot make sense as bounded operators from a function space to itself. This is the reason why by $\F$ we actually mean a family $\left(\F_r,\F^{(1)}_r\right)_{0<r<\infty}$ of normed spaces linked together by the dense inclusions $\F_R\subset \F_R^{(1)}\subset \F_r\subset \F_r^{(1)}$ if $r<R$ and each symbol $\D_s$ (as defined in (\ref{DS})) represents actually a family of bounded operators from $\F^{(1)}_r$ to $\F_r$. In particular, for any $r\in (0,\infty)$, $X(r)$ (where $X$ is the vector field in Theorem \ref{mainthm}) controls the norm of the operator $\D_s$ from $\F^{(1)}_r$ to $\F_r$. Once these required constructions are done we conclude the fourth section by the proof of Theorem \ref{mainthm}. In the last section we discuss briefly how to extend Theorem \ref{mainthm} to nonlinear systems of hyperbolic equations.\\

Some comments on our method: series expansions of solution to nonlinear \emph{ordinary} differential equations (ODE) have a long history and they have different formulations which are of course related, depending on their use. We can mention Lie series defined by K.T. Chen \cite{chen}, the Chen--Fliess series \cite{fliess} introduced in the framework of control theory by M. Fliess (or some variants like Volterra series or Magnus expansion \cite{magnus}) which are extensively used in control theory \cite{agrachev.gamkrelidze,sussman,kawskisussman} but also in the study of dynamical systems and in numerical analysis. Other major tools are Butcher series which explain the structure of Runge--Kutta methods of approximation of the solution of an ODE. They have been introduced by J.C. Butcher \cite{butcher} and developped by E. Hairer and G. Wanner \cite{hairerwanner} which explain that Runge--Kutta methods are gouverned by trees. Later on C. Brouder \cite{brouder,brouder2} realized that the structure which underlies the original Butcher's computation is exactly the Hopf algebra defined by D. Kreimer in his paper about the renormalization theory \cite{kreimer}. Concerning analogous results on nonlinear \emph{partial} differential equations, it seems that the fact that one can represent solutions or functionals on the set of solutions by series indexed by trees is known to physicists since the work of J. Schwinger and R. Feynman (and Butcher was also aware of that in his original work), although it is difficult to find precise references in the litterature (however for instance a formal series expansion is presented in \cite{duetsch}). However, to our knowledge, the only previous rigorous result (i.e. with a proof of convergence of the series) is the result in \cite{harrivel} already discussed. \\

\noindent \textbf{Notations} --- In the following we will denote by $M:= \R^{n+1}$ the Minkowski space. We fix a space-time splitting $M = \R\times \R^n$ and we note $t = x^0\in \R$ the time coordinate and $\vec{x} = (x^1,\cdots ,x^n)\in \R^n$ the space coordinates. We denote by $\square:= {\partial ^2\over \partial t^2} - \Delta$ the d'Alembertian, where $\Delta:= \sum_{i=1}^n {\partial ^2\over (\partial x^i)^2}$ is the Laplace operator. We define the Fourier transform on smooth fastly decreasing functions $f\in{\cal S}(\R^n)$ by
\[
\hat{f}(\xi) = {1\over \sqrt{2\pi}^n}\int_{\R^n} f(\vec{x})e^{-i\vec{x}\cdot \xi}d\vec{x}.
\]
And we extend it to the Schwartz class ${\cal S}'(\R^n)$ of tempered distributions by the standard duality argument. Then, for $s\in \R$, we let
\[
H^s(\R^n):= \{\varphi\in {\cal S}'(\R^n)|\ \epsilon^s\widehat{\varphi}\in L^2(\R^n)\},
\]
where $\epsilon(\xi) := \sqrt{|\xi|^2+m^2}$ and we set $||\varphi||_{H^s}:= ||\epsilon^s\widehat{\varphi}||_{L^2}$. We denote by $G$ the distribution in $\mathcal{C}^\infty(\R,{\cal S}'(\R^n))$ whose spatial Fourier transform is given by
\[
\widehat{G}(t,\xi) = {1\over \sqrt{2\pi}^n} {\sin \sqrt{|\xi|^2+m^2}t \over \sqrt{|\xi|^2+m^2}} = {1\over \sqrt{2\pi}^n} {\sin\left(\epsilon(\xi)t\right)  \over \epsilon(\xi) }.
\]
Note that $G\in \mathcal{C}^\ell(\R,H^{-s+1-\ell}(\R^n))$ (for $s>n/2$) is nothing but the fundamental solution of
\[
\square G + m^2G=0,
\]
with the initial conditions $G(0,\cdot) = 0$ and ${\partial G\over \partial t}(0,\cdot) = \delta_0$ that we already encountered.

\section{A formal description of the classical Fock space}\label{formalsection}
This part is essentially heuristic and provides the basic ideas which will become rigorous in the following sections.
\subsection{A formal algebra}\label{formalalgebra}
We first define a formal algebra ${\cal A}$ which helps us to define the generating function (\ref{magic}) (leaving aside the delicate question whether such an algebra exists). We let ${\cal A}$ be the algebra spanned over $\R$ by the symbols:
\[
\left(\phi^\tx(x),\phi^\td(y)\right)_{x,y\in M},
\]
where we assume the following properties.
\begin{description}
\item[(i)] $\forall x,y\in M$,  $[\phi^\td(y),\phi^\tx(x)] = G(y-x)$
\item[(ii)]  $\forall x,y\in M$, $[\phi^\tx(y),\phi^\tx(x)] = [\phi^\td(y),\phi^\td(x)] = 0$
\item[(iii)] $\phi^\td(x)$ and $\phi^\tx(x)$ depend smoothly on $x$ and $\square \phi^\tx +
  m^2 \phi^\tx = 0$ and $\square \phi^\td + m^2 \phi^\td = 0$.
\end{description}
We further require the existence of representations of ${\cal A}$ on
two vector spaces $\F$ and $\F^*$ with the following properties. We
denote by $|f\rangle$ elements in $\F$ and $\langle v|$ elements in $\F^*$ and, for any $\psi\in \mathcal{A}$ we write $|f\rangle\longmapsto \psi|f\rangle$ its action on $\F$ and $\langle v|\longmapsto \langle v|\psi$ its action on $\F^*$. We assume that there exists a pairing $\F^*\times \F\ni (\langle v|,|f\rangle)\longmapsto \langle v|f\rangle\in \R$ such that for any element $\psi\in {\cal A}$ on $\F^*$ we have $\forall \langle v|\in \F^*$, $\forall |f\rangle \in \F$, $\left(\langle v|\psi\right)|f\rangle = \langle v|\left(\psi|f\rangle\right) =: \langle w|\psi|f\rangle$.
Lastly we assume that there exist particular vectors $|0\rangle \in \F\setminus\{0\}$ and $\langle 0|\in
\F^*\setminus\{0\}$, such that
\begin{description}
\item[(iv)] $\forall y\in M$, $\phi^\td(y)|0\rangle = 0$
\item[(v)] $\forall x\in M$, $\langle 0|\phi^\tx(x) = 0$
\item[(vi)] $\langle0|0\rangle = 1$.
\end{description}
A simple consequence of properties (i) and (ii) is the following result, analogous to a special case of the `Wick theorem' used in quantum field theory.
\begin{lemm} For all $k\in \N^*$,
$\forall y_1,\cdots ,y_k\in M$, $\forall x\in M$,
\begin{equation}\label{wickbasic}
\left[\phi^\td(y_1)\cdots \phi^\td(y_k), \phi^\tx(x)\right] =  \sum_{j=1}^k G(y_j-x) \phi^\td(y_1)\cdots
\widehat{\phi^\td(y_j)} \cdots \phi^\td(y_k),
\end{equation}
where $\phi^\td(y_1)\cdots \widehat{\phi^\td(y_j)} \cdots \phi^\td(y_k):= \phi^\td(y_1)\cdots
\phi^\td(j_{j-1})\phi^\td(j_{j+1})\cdots \phi^\td(y_k)$.
\end{lemm}
{\em Proof} --- We prove (\ref{wickbasic}) by recursion on $k$. For
$k= 1$ this identity is nothing but (i). Assume that (\ref{wickbasic})
has been proved for $k$ points $y_j$'s, then, $\forall y_1,\cdots
,y_{k+1} \in M$, $\forall x\in M$, (by denoting $\phi^\td_j:=
\phi^\td(y_j)$ and $\phi^\tx:= \phi^\tx(x)$ for short)
\[
\begin{array}{ccl}
\phi^\td_1\cdots \phi^\td_{k+1} \phi^\tx & = & \displaystyle \phi^\td_1\left(\phi^\td_2\cdots \phi^\td_{k+1} \phi^\tx\right) =  \phi^\td_1\left[ \phi^\tx \phi^\td_2\cdots \phi^\td_{k+1}
  + \sum_{j=2}^{k+1} G(y_j-x) \phi^\td_2\cdots
\widehat{\phi^\td_j} \cdots \phi^\td_{k+1}\right]\\
& = & \displaystyle \phi^\tx\phi^\td_1\phi^\td_2\cdots \phi^\td_{k+1} + G(y_1-x)\phi^\td_2\cdots \phi^\td_{k+1}
+ \phi^\td_1 \sum_{j=2}^{k+1} G(y_j-x) \phi^\td_2\cdots
\widehat{\phi^\td_j} \cdots \phi^\td_{k+1}.
\end{array}
\]
and the result follows for $k+1$ points.\bbox

\subsection{Time invariance of the functional defined by (\ref{magic})}\label{timeinvariance}
We will see that, formally, the functional defined by (\ref{magic}) is time independent. We first introduce more concise notations, setting
\[
\int_t u\gd \phi^\td:=
\int_{y^0=t}u(y){\gd\over \partial y^0}\phi^\td(y)d\vec{y},
\]
\[
\int_0^t W(\phi^\tx,\nabla \phi^\tx)\phi^\td:=
\int_{0<z^0<t} W(\phi^\tx(z),\nabla \phi^\tx(z))\phi^\td(z)dz,
\]
and
\[
|f_\varphi\rangle:= \int_0\phi^\tx \gd \varphi|0\rangle := \int_{x^0=0}\phi^\tx(x){\gd\over \partial x^0}\varphi(x)d\vec{x}|0\rangle,
\]
so that we may write
\begin{equation}\label{Ftbasic}
\mathcal{F}^\varphi_t[u]_t = \langle 0|\ \hbox{exp}\left(\int_t u\gd \phi^\td\right)
T \hbox{exp}\left(-\int_0^t W(\phi^\tx,\nabla \phi^\tx)\phi^\td\right)|f_\varphi\rangle.
\end{equation}
As a preliminary we derive some relations which are consequences of (\ref{wickbasic}). It first implies that for all $k\in \N$,
\begin{equation}\label{wick1}
\left[\left(\int_tu\gd\phi^\td\right)^k,\phi^\tx(x)\right] = k \int_{y^0=t} u(y){\gd\over \partial y^0}G(y-x)\  \left(\int_tu\gd\phi^\td\right)^{k-1}.
\end{equation}
This follows from (\ref{wickbasic}) by writing
\[
\left(\int_tu\gd\phi^\td\right)^k = \int_{\{y_1^0 = \cdots = y_k^0 =t\}} u_1\cdots u_k\overleftrightarrow{\partial_1}\cdots\overleftrightarrow{\partial_k}\phi^\td_1\cdots \phi^\td_k\
d\vec{y}_1\cdots d\vec{y}_k,
\]
with the convention that $u_j:= u(y_j)$ and $\phi^\td_j:= \phi^\td(y_j)$. Then (\ref{wick1}) implies that for any analytic function of one variable $f$,
\begin{equation}\label{wick2}
\left[f\left(\int_tu\gd\phi^\td\right),\phi^\tx(x)\right] = \int_{y^0=t} u(y){\gd\over \partial y^0}G(y-x)\ f'\left(\int_tu\gd\phi^\td\right).
\end{equation}
But assuming that $x^0 = t$ and that $f(\zeta) = e^{\zeta}$, this gives
\begin{equation}\label{wick3}
\left[\hbox{exp}\left(\int_tu\gd\phi^\td\right),\phi^\tx(x)\right] = u(x)\  \hbox{exp}\left(\int_tu\gd\phi^\td\right)\quad \hbox{if }x^0=t,
\end{equation}
because of the Cauchy conditions on $G$ at time 0. However since $\langle 0|\phi^\tx(x) = 0$ (\ref{wick3}) implies in particular
\[
\langle0|\hbox{exp}\left(\int_tu\gd\phi^\td\right)\phi^\tx(x) = \langle0|u(x)\hbox{exp}\left(\int_tu\gd\phi^\td\right).
\]
We deduce from this relation that
\begin{equation}\label{wick4}
\langle0|\hbox{exp}\left( \int_tu\gd\phi^\td\right) \int_tW(\phi^\tx,\nabla \phi^\tx)\phi^\td = \langle0|\int_tW(u,\nabla u)\phi^\td\ \hbox{exp}\ \left(\int_tu\gd\phi^\td\right).
\end{equation}
Now assuming that the expression in (\ref{Ftbasic}) is differentiable with respect to time $t$,
\[
\begin{array}{ccl}
\displaystyle {d\mathcal{F}^\varphi_t[u]_t\over dt} & = &\displaystyle
\langle 0|{d\over dt}\left(\int_t u\gd \phi^\td\right)\hbox{exp}\left( \int_t u\gd \phi^\td\right)
T  \hbox{exp}\left(-\int_0^t W(\phi^\tx,\nabla \phi^\tx)\phi^\td\right)|f_\varphi\rangle \\
&  & \displaystyle - \langle 0|\ \hbox{exp}\left(\int_t u\gd \phi^\td\right) \left(\int_t W(\phi^\tx,\nabla \phi^\tx)\phi^\td\right) T \hbox{exp}\left(- \int_0^t W(\phi^\tx,\nabla \phi^\tx)\phi^\td\right)|f_\varphi\rangle.
\end{array}
\]
We hence deduce from (\ref{wick4}) that ${d\mathcal{F}^\varphi_t[u]_t\over dt}$ is equal to
\[
\langle 0|\left({d\over dt}\left(\int_t u\gd \phi^\td\right) - \int_t W(u,\nabla u)\phi^\td\right) \hbox{exp}\left(\int_t u\gd \phi^\td\right)
T  \hbox{exp}\left(-\int_0^t W(\phi^\tx,\nabla \phi^\tx)\phi^\td\right)|f_\varphi\rangle.
\]
But by using the fact that $\phi^\td$ is a solution of the linear Klein--Gordon equation and an integration by parts in space variables we obtain
\[
\begin{array}{ccl}
\displaystyle {d\over dt}\left(\int_t u\gd \phi^\td\right) = \int_t u\gd^2 \phi^\td & = & \displaystyle  \int_{\R^n}u{\partial^2\phi^\td\over \partial t^2} - {\partial^2u\over \partial t^2}\phi^\td = \int_{\R^n}u\left(\Delta \phi^\td - m^2\phi^\td\right) - {\partial^2u\over \partial t^2}\phi^\td\\
& = & \displaystyle \int_{\R^n}\left(\Delta u - m^2u\right)\phi^\td - {\partial^2u\over \partial t^2}\phi^\td = -\int_{\R^n} (\square u+ m^2u)\phi^\td.
\end{array}
\]
Hence ${d\mathcal{F}^\varphi_t[u]_t\over dt}$ is equal to
\[
\langle 0|\int_t (-\square u- m^2u- W(u,\nabla u))\phi^\td \hbox{exp}\left(\int_t u\gd \phi^\td\right)
T  \hbox{exp}\left(-\int_0^t W(\phi^\tx,\nabla \phi^\tx)\phi^\td\right)|f_\varphi\rangle,
\]
which vanishes if $\square u+ m^2u+ W(u,\nabla u) = 0$.

\subsection{Expanding the generating functions $\mathcal{F}_t[u]_t$}
The `Wick theorem' (\ref{wickbasic}) together with rules (iv) and (v) allow us also to recover an expansion of the functional $\mathcal{F}^\varphi_t[u]_t$ defined by (\ref{magic}) in terms of finite integrals defined by trees by using Feynman rules. To illustrate this we consider the simplest case, i.e. when $W(u,\nabla u) = \lambda u^2$, for some $\lambda \in \R$, and consider
\begin{equation}\label{serieexample}
\mathcal{F}^\varphi_t[u]_t = \langle 0|\ \hbox{exp}\left(\int_t u\gd \phi^\td\right)
T \hbox{exp}\left(-\int_0^t \lambda (\phi^\tx)^2\phi^\td\right)\int_0\phi^\tx\gd\varphi|0\rangle.
\end{equation}
First note that all expressions of the form $\langle0| \phi^\td_1\cdots \phi^\td_p\phi^\tx_1\cdots \phi^\tx_q|0\rangle$ (where $\phi^\td_i$, $\phi^\tx_j$ denote respectively $\phi^\td(y_i)$, $\phi^\tx(x_j)$ for some points $y_i,x_j\in M$) vanish unless $p = q$, as can be shown by repeated applications of (\ref{wickbasic}) and of (iv) and (v). Hence for instance the coefficient of the 0th power of $\lambda$ in the series (\ref{serieexample}) is
\[
\mathcal{F}^\varphi_{t,1}[u]_t = \langle 0|\ \hbox{exp}\left(\int_t u\gd \phi^\td\right) 1 \int_0\phi^\tx\gd\varphi|0\rangle
 = \langle 0| \int_t u_y\overleftrightarrow{\partial_y} \phi^\td_y\int_0\phi^\tx_x\overleftrightarrow{\partial_x}\varphi_x|0\rangle,
\]
where we set $u_y:= u(y)$, $\overleftrightarrow{\partial_y}:= {\gd\over \partial y^0}$, $\phi^\td_y:= \phi^\td(y)$, $\phi^\tx_x:= \phi^\tx(x)$, etc. However it follows from (i), (iv) and (v) that $\langle 0|\phi^\td_y\phi^\tx_x|0\rangle = \langle 0|G_{yx} + \phi^\tx_x\phi^\td_y|0\rangle = G_{yx}$ (where $G_{yx}:= G(y-x)$). Hence
\[
\mathcal{F}^\varphi_{t,1}[u]_t = \int_{y^0=t}\int_{x^0=0} u_y\overleftrightarrow{\partial_y} G_{yx}\overleftrightarrow{\partial_x}\varphi_x = I^\varphi_t[u]_t.
\]
The coefficient of the first power of $\lambda$ is
\[
\begin{array}{ccl}
\mathcal{F}^\varphi_{t,2}[u]_t^{\otimes 2} & = & \displaystyle \langle 0|\ \hbox{exp}\left(\int_t u\gd \phi^\td\right)
\left(-\int_0^t  (\phi^\tx)^2\phi^\td\right)\int_0\phi^\tx\gd\varphi|0\rangle\\
 & = & \displaystyle - \langle 0|{1\over 2} \int_t\int_tu_1u_2\overleftrightarrow{\partial_1} \overleftrightarrow{\partial_2} \phi_1^\td\phi_2^\td\int_0^t (\phi^\tx_z)^2\phi^\td_z\int_0\phi^\tx_x\overleftrightarrow{\partial_x}\varphi|0\rangle,
\end{array}
\]
where, for $a = 1,2$, $u_a:= u(y_a)$, $\overleftrightarrow{\partial_a}:= {\gd\over \partial y_a^0}$, $\phi_a^\td:= \phi^\td(y_a)$ and $\phi^\tx_z:= \phi^\tx(z)$, etc. Denoting also by $G_{az}:= G(y_a-z)$, for $a=1,2$, we have by (\ref{wickbasic}) $[\phi^\td_1\phi^\td_2,\phi^\tx_z] = G_{1z}\phi^\td_2 + G_{2z}\phi^\td_1$, which implies by (v) that $\langle 0|\phi^\td_1\phi^\td_2\phi^\tx_z = G_{1z}\langle 0|\phi^\td_2
 + G_{2z}\langle 0|\phi^\td_1$. A second application of (\ref{wickbasic}) and (v) gives then
\[
\langle 0|\phi^\td_1\phi^\td_2\left(\phi^\tx_z\right)^2 = G_{1z}\langle 0|\phi^\td_2 \phi^\tx_z
 + G_{2z}\langle 0|\phi^\td_1\phi^\tx_z = \left(G_{1z}G_{2z} + G_{2z}G_{1z}\right) \langle 0|.
\]
We hence deduce that
\[
\begin{array}{ccl}
\mathcal{F}^\varphi_{t,2}[u]_t^{\otimes 2} & = & \displaystyle  - \int_t\int_tu_1u_2\overleftrightarrow{\partial_1} \overleftrightarrow{\partial_2}\int_0^tG_{1z}G_{2z}\langle 0| \phi^\td_z\int_0\phi^\tx_x\overleftrightarrow{\partial_x}\varphi|0\rangle, \\
& = & \displaystyle  -\int_t\int_tu_1u_2\overleftrightarrow{\partial_1} \overleftrightarrow{\partial_2}\int_0^tG_{1z}G_{2z}\int_0G_{zx}\overleftrightarrow{\partial_x}\varphi,
\end{array}
\]
where we set $G_{zx}:= G(z-x)$ and we further use (\ref{wickbasic}) and (iv). We hence recover the same expression for $\mathcal{F}^\varphi_{t,2}[u]_t^{\otimes 2}$ as the one obtained in \cite{harrivel} and expounded in the introduction.

\subsection{How to construct the representation of $\mathcal{A}$ on $\F$ and $\F^*$}\label{howto}

The heuristic idea to construct ${\cal A}$ and its representation $\F$ is the following. We still denote by ${\cal E}_0$ the space of solutions of the linear Klein--Gordon equation (\ref{kg}) $\square \varphi + m^2\varphi = 0$ on $M$ and we let $\F$ to be the set of analytic functionals $|f\rangle:{\cal E}_0\longrightarrow \R$, i.e. such that
\begin{equation}\label{definitionfphi}
\forall \varphi\in \mathcal{E}_0,\quad
|f\rangle(\varphi) = \sum_{p=0}^\infty |f_p\rangle(\varphi\otimes \cdots \otimes \varphi) = \sum_{p=0}^\infty |f_p\rangle\left(\varphi^{\otimes p}\right),
\end{equation}
where we can view each $|f_p\rangle$ as a symmetric $p$-multilinear functional on $\left(\mathcal{E}_0\right)^p$ or, alternatively, as a linear form on $\left(\mathcal{E}_0\right)^{\otimes p}$.
Lastly we let ${\cal A}$ to be the algebra of linear operators acting on $\F$ and $\F^*$ the dual space of $\F$.
\begin{itemize}
\item \textbf{The definition of $\phi^\tx(x)$} --- To any point $x\in M$ we associate a particular (linear) functional on $\mathcal{E}_0$, the evaluation at $x$:
\[
\begin{array}{cccc}
|x\rangle:& \mathcal{E}_0 & \longrightarrow & \R\\
 & \varphi & \longmapsto & |x\rangle(\varphi):= \varphi(x).
\end{array}
\]
And we let $\phi^\tx(x)$ to be the operator of multiplication of functionals in $\F$ by the functional $|x\rangle$:
\[
\begin{array}{cccc}
\phi^\tx(x):& \F & \longrightarrow & \F\\
 & |f\rangle & \longmapsto & |x\rangle|f\rangle,
\end{array}
\]
i.e. $\forall \varphi\in \mathcal{E}_0$, $\left(\phi^\tx(x)|f\rangle\right)(\varphi) = \varphi(x)\left(|f\rangle(\varphi)\right)$.
\item \textbf{The definition of $\phi^\td(y)$} --- We remark that, for any $y\in M$, the distribution $\Gamma_y:x\longmapsto G(y-x)$ is a weak solution of the linear Klein--Gordon equation, hence an element of $\mathcal{E}_0$. Thus, for any $|f\rangle\in \F$, we can formally define the variational derivative
\[
{\delta |f\rangle\over \delta \Gamma_y}(\varphi):=
\lim_{\varepsilon\rightarrow 0} {|f\rangle(\varphi + \varepsilon
  \Gamma_y) - |f\rangle(\varphi)\over \varepsilon}.
\]
and we let
\[
\begin{array}{cccc}
\phi^\td(y):& \F & \longrightarrow & \F\\
 & |f\rangle & \longmapsto & {\delta |f\rangle\over \delta \Gamma_y}.
\end{array}
\]
\end{itemize}
An alternative, more algebraic definition of $\phi^\td(y)$ for $|f\rangle$ given by (\ref{definitionfphi}) is
\[
\left(\phi^\td(y)|f\rangle\right)(\varphi) = \sum_{p=1}^\infty p\,|f_p\rangle(\Gamma_y\otimes \varphi\otimes \cdots \otimes \varphi).
\]
Then we observe that by using the Leibniz rule
\[
\phi^\td(y)\phi^\tx(x)|f\rangle =
{\delta (\phi^\tx(x)|f\rangle)\over \delta \Gamma_y} =
{\delta (|x\rangle|f\rangle)\over \delta \Gamma_y} = {\delta|x\rangle \over \delta
  \Gamma_y} |f\rangle + |x\rangle {\delta |f\rangle\over \delta
  \Gamma_y}.
\]
But since $|x\rangle$ is a linear functional ${\delta|x\rangle \over \delta
  \Gamma_y} = |x\rangle(\Gamma_y) = G(y-x)$, hence
\[
\phi^\td(y)\phi^\tx(x)|f\rangle = G(y-x)|f\rangle + \phi^\tx(x) {\delta |f\rangle\over \delta \Gamma_y} = G(y-x)|f\rangle + \phi^\tx(x)\phi^\td(y)|f\rangle.
\]
We thus deduce that
\[
\left[ \phi^\td(y),\phi^\tx(x)\right] = G(y-x).
\]
Moreover it is obvious that (still at a formal level) $[\phi^\td(y),\phi^\td(y')] =
[\phi^\tx(x),\phi^\tx(x')] = 0$. Lastly we let $|0\rangle\in \F$ to be the
constant functional ${\cal E}_0\ni \varphi \longmapsto 1$ and
$\langle 0|\in\F^*$ to be
\[
\begin{array}{cccc}
\langle 0|: & \F & \longrightarrow & \R\\
& |f\rangle & \longmapsto & |f\rangle(0),
\end{array}
\]
i.e. $\langle 0|$ plays the role of the Dirac mass at $0\in \mathcal{E}_0$. Then obviously (iv), (v) and (vi) are satisfied. In this representation we remark that an interpretation of $\langle [u]_t|$ in (\ref{magic}) is possible. Indeed if we think the operator $\int_tu\gd\phi^\td$ as a constant vector field on the infinite dimensional space $\mathcal{E}_0$, then the exponential $\hbox{exp}\left(\int_tu\gd\phi^\td\right)$ acts on $\F$ by translation, i.e. through $\left(\hbox{exp}\left(\int_tu\gd\phi^\td\right)|f\rangle\right)(\varphi) = |f\rangle(\varphi + \varphi_{[u]_t})$, where $\varphi_{[u]_t}\in \mathcal{E}_0$ is actually the solution of (\ref{kg}) which has the same Cauchy data at time $t$ as $u$ (in the notations of Paragraph \ref{paragraph2.2.2}, $\varphi_{[u]_t}:= u\sharpgdt G$). Hence we deduce that $\langle [u]_t|:= \langle 0|\hbox{exp}\left(\int_tu\gd\phi^\td\right)$ is the Dirac mass at $\varphi_{[u]_t}\in \mathcal{E}_0$. These facts will be proved in Corollary \ref{coroexp}.

\subsection{Towards a well-defined theory}
All the preceding constructions are completely formal, as long as we do not precise any topology on $\F$, $\F^*$ and $\mathcal{A}$. Moreover we would like to find topologies in such a way that the operators $\phi^\tx$ and $\phi^\td$ are simultaneously well-defined. Here the main difficulty occurs, indeed:
\begin{itemize}
\item[(a)] on the one hand, in order to define $\phi^\tx(x)$, we need to make sense of $|x\rangle: \varphi\longmapsto \varphi(x)$ as a continuous operator and this requires the functions $\varphi$ in $\mathcal{E}_0$ to be continuous. Actually a careful inspection of the formal computations done in section \ref{timeinvariance} reveals that one also needs to define ${\partial \phi^\tx\over \partial x^0}$, which means that we actually need that $\varphi$ be of class $\mathcal{C}^1$ with respect to time;
\item[(b)] on the other hand the definition of $\phi^\td(y)$ is consistent if we can differentiate with respect to $\Gamma_y$, i.e. if $\Gamma_y\in \mathcal{E}_0$. However $\Gamma_y$ is only a distribution and hence this is in conflict with the first requirement.
\end{itemize}
The key observation to avoid these difficulties is that, in the definition of (\ref{magic}) and in the formal computations done in section \ref{timeinvariance} we only need to define
\[
\U(t):= \int_tu\gd\phi^\td,\quad  \V(t):= -\int_t(\square u + m^2u)\phi^\td\quad
\hbox{and} \quad \D(t):= \int_tW(\phi^\tx,\nabla \phi^\tx)\phi^\td
\]
as continuous maps of the time $t$, with values in a set of continuous operators and to assume that $\U$ is derivable with respect to time, with ${d\U\over dt} = \V$. Hence we will assume that functions in $\mathcal{E}_0$ are sufficiently smooth, so that operators $\phi^\tx$ and ${\partial \phi^\tx\over \partial t}$ will be well-defined. However we will be able to make sense of $\U$ and $\V$ if $u$ is sufficiently smooth and to define $\D$ again if functions in $\mathcal{E}_0$ are sufficiently smooth. Note that, on the infinite dimensional manifold $\mathcal{E}_0$, $\U$ and $\V$ can be viewed as tangent vector fields with constant coefficients, whereas $\D$ is a tangent vector field with non constant analytic coefficients.


\section{The analytic setting}\label{sectionsetting}
In this section we define precisely the `Fock space' $\F$ and introduce the operators $\U$, $\V$ and $\D$ as well as the exponential of these operators in order to make sense of \refeq{magic}.

\subsection{Functions spaces}\label{sub:functions_spaces} 
First we define more precisely the space ${\cal E}_0$ of solution of the linear Klein--Gordon equation \refeq{kg}, in particular we describe the topology of this space. For $s\in\R$ we define ${\cal E}^{s+1}_0$ by
\[
\mathcal{E}_0^{s+1}:= \{\varphi\in \mathcal{C}^0(\R,H^{s+1}(\R^n))\cap \mathcal{C}^1(\R,H^s(\R^n))|\ \square \varphi + m^2\varphi = 0\}.
\]
Each map $\varphi\in \mathcal{E}_0^{s+1}$ is characterized by its Cauchy data $[\varphi]_0 = (\varphi_0,\varphi_1)\in H^{s+1}(\R^n)\times H^s(\R^n)$ at time $t=0$. Indeed one recovers $\varphi$ from $(\varphi_0,\varphi_1)$ through the relation
\[
\hat{\varphi}(t,\xi) = \hat{\varphi}_0(\xi)\cos \epsilon(\xi)t +
\hat{\varphi}_1(\xi){\sin \epsilon(\xi)t\over \epsilon(\xi)}\quad \quad (\epsilon(\xi):= \sqrt{|\xi|^2+m^2}),
\]
where $\hat{\varphi}_0$ and $\hat{\varphi}_1$ are the spatial Fourier transform of respectively $\varphi_0$ and $\varphi_1$. Note that the quantity
\[
||\varphi||^2_{{\cal E}^{s+1}}:= ||\varphi(t,\cdot)||_{H^{s+1}}^2 + ||{\partial\varphi\over \partial t}(t,\cdot)||_{H^s}^2
\]
is independant of $t\in \R$.
\begin{defi}
For each $p\in \N$ we let $(\bigotimes^p\mathcal{E}_0^{s+1})^*$ be the space of linear applications $|f_p\rangle:\bigotimes^p\mathcal{E}_0^{s+1}:= \mathcal{E}_0^{s+1}\otimes \cdots \otimes \mathcal{E}_0^{s+1} \longrightarrow \R$ which are continuous, i.e. such that there exists a constant $C>0$ such that $\forall \varphi_1,\cdots ,\varphi_p\in \mathcal{E}_0^{s+1}$, $||f_p\rangle(\varphi_1\otimes \cdots \otimes \varphi_p)| \leq C\,
 ||\varphi_1||_{{\cal E}^{s+1}_0}\cdots ||\varphi_p||_{{\cal E}^{s+1}_0}$. We denote by $\lc f_p\rf$ the optimal value of $C$ in this inequality, so that we have: 
\begin{equation}\label{Fcontinuous}
||f_p\rangle(\varphi_1\otimes \cdots \otimes \varphi_p)| \leq \lc f_p\rf\,
 ||\varphi_1||_{{\cal E}^{s+1}_0}\cdots ||\varphi_p||_{{\cal E}^{s+1}_0}, \quad
 \forall \varphi_1,\cdots ,\varphi_p\in \mathcal{E}_0^{s+1}.
\end{equation}
We let $S(\bigotimes^p\mathcal{E}_0^{s+1})^*$ be the subspace of $|f_p\rangle\in (\bigotimes^p\mathcal{E}_0^{s+1})^*$ which are symmetric, i.e. for all permutation $\sigma\in \mathfrak{S}_p$,
\begin{equation}\label{symmetric}
|f_p\rangle(\varphi_{\sigma(1)}\otimes \cdots \otimes \varphi_{\sigma(p)}) = |f_p\rangle(\varphi_1\otimes \cdots \otimes \varphi_p),\quad
\forall \varphi_1,\cdots ,\varphi_p\in \mathcal{E}_0^{s+1}.
\end{equation}
\end{defi}
We can now define our `Fock space':
\begin{defi} For any $s\in \R$ and $r\in (0,+\infty)$, we let $\F_r$ be the space of formal series
\[
|f\rangle = \sum_{p=0}^\infty |f_p\rangle,
\]
where each $|f_p\rangle\in S(\bigotimes^p\mathcal{E}_0^{s+1})^*$ and such that the quantity
\begin{equation}\label{norme}
N_r(|f\rangle):= \sum_{p=0}^\infty \lc f_p\rf r^p
\end{equation}
is finite. Then $\left( \F_r,N_r\right)$ is a Banach space.
\end{defi}
Actually we can identify any $|f_p\rangle\in
S(\bigotimes^p\mathcal{E}_0^{s+1})^*$ with a continuous homogeneous
polynomial map of degree $p$ from $\mathcal{E}^{s+1}_0$ to $\R$ by
the relation $|f_p\rangle(\varphi):= |f_p\rangle(\varphi\otimes \cdots \otimes
\varphi) = |f_p\rangle(\varphi^{\otimes p})$ and hence series $|f\rangle \in \F_r$ with convergent analytic series on $B_{{\cal E}^{s+1}_0}(0,r)$, the ball of radius $r$ in
$\mathcal{E}_0^{s+1}$, by the relation
\[
|f\rangle (\varphi) = \sum_{p=0}^\infty |f_p\rangle(\varphi^{\otimes p}).
\]
Let us define $|0\rangle$ to be the constant functional over ${\cal E}^{s+1}_0$ which is equal to $1$:
\[
    |0\rangle:\varphi\in{\cal E}^{s+1}_0\longmapsto 1.
\]
Then $|0\rangle$ belongs to $\F_r$ for all $r>0$ and $N_r(|0\rangle)=1$. \\

\noindent
We will also extend the definition
(\ref{norme}) to a complex variable $z$:
\[
N_z(|f\rangle):= \sum_{p=0}^\infty \lc f_p\rf z^p,
\quad \forall z\hbox{ such that }|z|\leq r,
\]
which gives us, for any fixed $|f\rangle\in \F_r$, a holomorphic
function on the ball $B_\C(0,r)\subset \C$. We define, for $k\in \N$,
\[
N_r^{(k)}(|f\rangle):= {d^k\over dz^k}N_z(|f\rangle)|_{z=r} \quad
\hbox{and} \quad  \F_r^{(k)}:= \{|f\rangle\in \F_r|\
N_r^{(k)}(|f\rangle) < +\infty\},
\]
so that for instance $N_r^{(1)}(|f\rangle):= \sum_{p=1}^\infty
p\lc f_p\rf r^{p-1}$. We set $\F_\infty:= \cap_{r>0}\F_r$ and
$\F_{pol}:= \{|f\rangle = \sum_{p=0}^N|f_p\rangle |\ N\in \N,
|f_p\rangle\in S(\bigotimes^p\mathcal{E}_0^{s+1})^*\}$. Note that,
using in particular the obvious inequality $r<R\ \Longrightarrow \
N_r(\varphi) < N_R(\varphi)$, we have the dense inclusions
\[
\forall r,R\in (0,\infty), \hbox{ s.t. }r<R, \forall k\in \N,\quad
\F_{pol} \subsetneq \F_\infty \subsetneq \F^{(k+1)}_R\subsetneq \F^{(k)}_R
\subset \F_R \subsetneq \F_r.
\]
Since the space $\F_r$ is an (infinite dimensional) topological vector space for all $r>0$, we can consider its topological dual space, denoted by $(\F_r)^*$. For example we can consider the linear form $\langle 0|$ over $\F_r$ defined by
\[
\langle 0| : \F_r\ni |f\rangle \longmapsto \langle 0|f\rangle:=|f\rangle(0),
\]
Then one can easily see that $\langle 0|$ belongs to $\cap_{r>0}(\F_r)^*$ and that we have $\langle 0|0\rangle=1$. More generally if $u$ belongs to $\cap_{\ell=0,1}{\cal C}^\ell((\underline{t},\overline{t}),H^{s+1-\ell}(\R^n))$ then for $t\in(\underline{t},\overline{t})$ we can consider $\varphi_{[u]_t}$ the (unique) element of ${\cal E}^{s+1}_0$ such that $[\varphi_{[u]_t}]_t=[u]_t$ (i.e. $\varphi_{[u]_t}$ is the solution of the linear Klein--Gordon equation \refeq{kg} with the same Cauchy data at time $t$ as $[u]_t$). Then since $||\varphi_{[u]_t}||_{{\cal E}^{s+1}_0}=||[u]_t||_{s+1}$, we can consider $|f\rangle(\varphi_{[u]_t})$ to be the evaluation of the functional $|f\rangle$ on $\varphi_{[u]_t}$, for all $r>0$ such that $r>||[u]_t||_{s+1}$ and, for all $|f\rangle\in\F_r$. This define a linear form $\langle [u]_t|$ over $\F_r$ :
\[
    \forall r>||[u]_t||_{s+1};\quad \langle [u]_t|:\F_r\ni |f\rangle\longmapsto \langle [u]_t|f\rangle:=|f\rangle(\varphi_{[u]_t}).
\]
Then $\langle [u]_t|$ belongs to $\cap_{r>||[u]_t||_{s+1}} (\F_r)^*$.
\begin{defi}
For any $r_0\in (0,\infty]$ and any $k,\ell\in \N$ a \textbf{continuous operator $\T$ from $\F_{(0,r_0)}^{(k)}$ to $\F_{(0,r_0)}^{(\ell)}$} is a family $\left( \T_r\right)_{0<r<r_0}$, where, for any $r\in (0,r_0)$, $\T_r:\F_r^{(k)}\longrightarrow \F_r^{(\ell)}$ is a continuous linear operator with norm $||\T_r||$ and such that, $\forall r,r'\in (0,r_0)$, if $r<r'$, then the restriction of $\T_r$ to $\F_{r'}^{(k)}$ coincides with $\T_{r'}$. Moreover, if $X:(0,r_0)\longrightarrow (0,\infty)$ is a locally bounded function, we say that \textbf{the norm of $\T$ is controlled by $X$} if, $\forall r\in (0,r_0)$, $||\T_r||\leq X(r)$.\\
For simplicity we systematically denote each operator $\T_r$ by $\T$ in the following.
In the case where $r_0=\infty$, we will just write that \textbf{$\T$ is a continuous operator from $\F^{(k)}$ to $\F^{(\ell)}$}.
\end{defi}
The following result concerns an example of a continuous operator from $\F^{(1)}$ to $\F$ with a norm controlled by the constant function $||\psi||_{\mathcal{E}^{s+1}_0}$.
\begin{lemm}\label{Ldeltapsi}
Let $s\in \R$ and $\psi\in \mathcal{E}_0^{s+1}$. Then for any $r\in (0,\infty)$ and $|f\rangle\in \F_r^{(1)}$, the functional ${\delta |f\rangle\over \delta \psi}$ defined by
\[
\forall \varphi\in B_{{\cal E}^{s+1}_0}(0,r),\quad
{\delta |f\rangle\over \delta \psi}(\varphi):= \sum_{p=1}^\infty p|f_p\rangle(\psi\otimes \varphi^{\otimes p-1})
\]
belongs to $\F_r$ and
\begin{equation}\label{deltapsi}
N_r\left({\delta |f\rangle\over \delta \psi}\right) \leq ||\psi||_{\mathcal{E}^{s+1}_0} N_r^{(1)}(|f\rangle).
\end{equation}
\end{lemm}
\begin{proof}
    For any $\varphi\in \mathcal{E}_0^{s+1}$ we have
    \[
    \forall p\in \N,\quad \left| p|f_p\rangle\left(\psi\otimes \varphi^{\otimes p-1} \right)\right| \leq p\lc f_p\rf  ||\psi||_{\mathcal{E}^{s+1}_0} ||\varphi||^{p-1}_{\mathcal{E}^{s+1}_0}.
    \]
    Hence if we denote $|g\rangle:= {\delta |f\rangle\over \delta \psi}$ we deduce that $\lc g_{p-1}\rf \leq p||\psi||_{\mathcal{E}^{s+1}_0} \lc f_p\rf $ and so
    \[
    N_r\left({\delta |f\rangle\over \delta \psi}\right) = \sum_{p=1}^\infty \lc g_{p-1}\rf r^{p-1} \leq \sum_{p=1}^\infty p||\psi||_{\mathcal{E}^{s+1}_0} \lc f_p\rf r^{p-1} = ||\psi||_{\mathcal{E}^{s+1}_0} N_r^{(1)}(|f\rangle).
    \]
\end{proof}

\subsection{Definition of the operators}
Now we define the creation and annihilation operators and derive some basic properties from the definitions.
\subsubsection{Creation operators}
For any $t\in \R$, $\phi_1\in H^{-s-1}(\R^n)$ we define $\int_t|x\rangle \phi_1(x) \in \F_{pol}$ to be the (linear) functional
\[
\begin{array}{cccc}
\displaystyle \int_t|x\rangle \phi_1(x) : & \mathcal{E}^{s+1}_0 & \longrightarrow & \R\\
& \varphi &\longmapsto & \displaystyle \int_{\R^n}\varphi(t,\vec{x}) \phi_1(\vec{x}) d\vec{x},
\end{array}
\]
which is obviously continuous with $N_r(\int_t|x\rangle \phi_1(x)) = r||\phi_1||_{H^{-s-1}}$, $\forall r>0$. For example, if $s>n/2$ and $\phi_1= \delta_{\vec{x}}$, then $\int_t|x\rangle \phi_1(x) = |x\rangle:\varphi\longmapsto \varphi(x)$, where $x=(t,\vec{x})$ (see Lemma \ref{lemmapp1}). Similarly we define, for $\phi_0\in H^{-s}(\R^n)$,
\[
\begin{array}{cccc}
\displaystyle \int_t{\partial |x\rangle\over \partial t} \phi_0(x) : & \mathcal{E}^{s+1}_0 & \longrightarrow & \R\\
& \varphi &\longmapsto & \displaystyle \int_{\R^n}{\partial \varphi\over \partial t}(t,\vec{x}) \phi_0(\vec{x}) d\vec{x},
\end{array}
\]
with $N_r(\int_t{\partial |x\rangle\over \partial t} \phi_0(x)) = r||\phi_0||_{H^{-s}}$, $\forall r>0$. If we take $s>n/2$ and $\phi_0= \delta_{\vec{x}}$, then $\int_t{\partial|x\rangle\over \partial t} \phi_1(x) = {\partial|x\rangle\over \partial t}:\varphi\longmapsto {\partial\varphi\over \partial t}(x)$, where $x=(t,\vec{x})$. This leads us to the definition of the operators
\[
\int_t\phi^\tx \phi_1:|f\rangle\longmapsto \left(\int_t|x\rangle \phi_1(x)\right) |f\rangle
\quad\hbox{and}\quad
\int_t{\partial \phi^\tx\over \partial t} \phi_0:|f\rangle\longmapsto \left(\int_t{\partial |x\rangle\over \partial t} \phi_0(x)\right) |f\rangle,
\]
which are clearly continuous operators from $\F$ to $\F$ with norm controlled by $r||\phi_1||_{H^{-s-1}}$ and $r||\phi_0||_{H^{-s}}$ respectively. Hence, for any $\varphi\in \mathcal{E}^{s+1}_0$ and $r>0$, we can define
\[
\int_t\phi^\tx\gd \varphi:= \int_t{\partial \phi^\tx\over \partial t}\varphi -  \int_t\phi^\tx {\partial \varphi\over \partial t}: \F_r\longmapsto \F_r.
\]

\subsubsection{Annihilation operators}\label{paragraph2.2.2}
For any $q\in \R$, $k\in \N$, $t\in \R$ and for any function $v\in
H^q(\R^n)$ we denote by $v\sharp_tG^{(k)}$ the distribution on
$\R\times \R^n$ defined by
\begin{equation}\label{vsharpG}
\forall x = (x^0,\vec{x})\in \R^{n+1},\quad v\sharp_tG^{(k)}(x):= \int_{\R^n}
v(\vec{y}){\partial^kG\over \partial t^k} (t-x^0, \vec{y}-\vec{x})
d\vec{y}.
\end{equation}
The various properties of $v\sharp_tG^{(k)}$ are derived in Lemma \ref{lemmapp3}. They imply that, if $u\in
\mathcal{C}^\ell((\underline{t},\overline{t}),H^{s+1-\ell}(\R^n))$,
then $\forall t\in (\underline{t},\overline{t})$,
$\left({\partial^\ell u\over \partial t^\ell}|_t\right)\sharp_t
G^{(k)} \in \mathcal{E}_0^{s+2-k-\ell}(\R^n)$ with the identity
\begin{equation}\label{ulsharpGk}
  ||\left({\partial^\ell u\over \partial t^\ell}|_t \right)\sharp_t G^{(k)}||_{\mathcal{E}^{s+2-k-\ell}}
 = ||{\partial^\ell u\over \partial t^\ell}|_t||_{H^{s+1-\ell}}.
\end{equation}
Now, for $u\in \cap_{\ell=0,1}\mathcal{C}^\ell((\underline{t},\overline{t}),H^{s+1-\ell}(\R^n))$, we let
\[
u\sharpgdt G:=
\left(u|_t \right)\sharp_t G^{(1)} - \left({\partial u\over \partial t}|_t \right)\sharp_tG \in \mathcal{E}_0^{s+1}.
\]
Note that, alternatively, $u\sharpgdt G$ could be defined as the unique element in $\mathcal{E}_0^{s+1}$ which shares the same Cauchy data at time $t$ as $u$. Thus we can define operator $\int_tu\gd\phi^\td$ (also denoted by $\U(t)$) by the following
\begin{equation}\label{defop2}
 \U(t)=\int_tu\gd\phi^\td:= {\delta\over \delta \left(u\sharpgdt G\right)}.
\end{equation}
This operator satisfies the :
\begin{lemm}
    Let $u$ belong to $\cap_{\ell=0,1}\mathcal{C}^\ell((\underline{t},\overline{t}),H^{s+1-\ell}(\R^n))$ then for all $t\in(\underline{t},\overline{t})$, the operator $\U(t)=\int_tu\gd\phi^\td$ given by \refeq{defop2} is a continuous operator from $\F^{(1)}$ to $\F$. Moreover the norm of $\U(t)$ is controlled by the constant  $||[u]_t||_{s+1}$ (we recall that we have set $||[u]_t||_{s+1}:=||u|_{t}||_{H^{s+1}}+||{\partial u \over \partial t}|_t||_{H^s}$).
\end{lemm}
\begin{proof}
    By using (\ref{ulsharpGk}) we get $||\left({\partial u\over \partial t}|_t \right)\sharp_tG||_{\mathcal{E}^{s+1}} = ||{\partial u\over \partial t}|_t||_{H^{s}}$ and $||\left(u|_t \right)\sharp_t G^{(1)}||_{\mathcal{E}^{s+1}} =   ||u|_t||_{H^{s+1}}$, so that $||u\sharpgdt G||_{\mathcal{E}^{s+1}} =   ||[u]_t||_{s+1}$.   Thus it follows from Lemma \ref{Ldeltapsi} that, for all $t\in(\underline{t},\overline{t})$ and for all $r>0$, $\int_t u \gd\but : \F_r^{(1)}\longrightarrow \F_r$ is a well defined bounded operator and for all $t\in(\underline{t},\overline{t})$
    \begin{equation}\label{estop1}
    \forall |f\rangle\in\F_r^{(1)};\quad N_r\left(\int_t u \gd \but
    |f\rangle\right)\le ||[u]_t||_{s+1} N_r^{(1)}(|f\rangle)
    \end{equation}
    which completes the proof.
\end{proof}

\noindent We now look in which circumstances one can define 
\[
\V(t) = -\int_t(\square u+m^2u)\phi^\td:= -{\delta\over \delta \left((\square u+m^2u)|_t\sharp_t G\right)}
\]
and show that ${d\U\over dt}(t) = \V(t)$.
\begin{lemm}\label{lemme.u}
Let $u\in \mathcal{C}^0((\underline{t},\overline{t}),H^{s+1}(\R^n))
\cap \mathcal{C}^1((\underline{t},\overline{t}),H^s(\R^n))$,
assume that there exists some $J\in
\mathcal{C}^0((\underline{t},\overline{t}),H^s(\R^n))$ such that
\begin{equation}\label{weaksol}
\square u + m^2u = - J\quad \hbox{in the distribution sense on }(\underline{t},\overline{t})\times \R^n
\end{equation}
and set
\begin{equation}\label{Jphi}
\int_tJ\phi^\td :=
{\delta \over \delta \left(\left(J|_t\right)\sharp_t G\right)}.
\end{equation}
Then the map $t \longmapsto u\sharpgdt G$ from
$(\underline{t},\overline{t})$ to $\mathcal{E}_0^{s+1}$ is of class
$\mathcal{C}^1$ and
\begin{equation}\label{Jsharp}
{d \over dt}\left( u\sharpgdt G\right) = \left(J|_t\right)\sharp_tG\quad \hbox{in }\mathcal{E}_0^{s+1}.
\end{equation}
As a consequence we have
\begin{equation}\label{delta/deltaphi}
{d \over dt}\left(\int_t u \gd \but\right) = \int_tJ\phi^\td.
\end{equation}
\end{lemm}
\begin{proof} We only need to show (\ref{Jsharp}) in the distribution sense: this
will imply that (\ref{Jsharp}) holds also strongly since
$[t\longmapsto \left(J|_t\right)\sharp_tG]$ is a continuous map into
$\mathcal{E}^{s+1}_0$. Thus let $\chi\in
\mathcal{C}^\infty_0((\underline{t},\overline{t}))$ and $x\in
\R^{n+1}$ and let us compute
\[
\begin{array}{ccl}
\displaystyle \int_{\underline{t}}^{\overline{t}} dt\chi'(t) \left( u\sharpgdt G\right)(x) & = &
\displaystyle \int_{\underline{t}}^{\overline{t}} dt \chi'(t)
\int_{\R^n}d\vec{y}\left( u(t,\vec{y}){\partial G\over \partial
t}(t-x^0,\vec{y}-\vec{x}) \right.\\
& & \displaystyle \quad \quad \quad \quad \quad \quad \quad \quad
\quad \left. - {\partial u\over
\partial t}(t,\vec{y})G(t-x^0,\vec{y}-\vec{x}) \right)\\ & = & \displaystyle
\int_{\R^{n+1}}dy\chi'(y^0) \left( u(y){\partial G\over \partial
t}(y-x) - {\partial u\over
\partial t}(y)G(y-x)\right),
\end{array}
\]
where we have set $y = (y^0,\vec{y}) = (t,\vec{y})$. Now we observe
that, by denoting $G_x(y):= G(y-x)$,
\[
\chi'(y^0) \left( u{\partial G_x\over \partial
t} - {\partial u\over
\partial t}G_x \right)(y) = \left({\partial (\chi u)\over
\partial y^0}{\partial G_x\over \partial t} - {\partial u\over \partial t} {\partial (\chi G_x)\over \partial y^0}\right)(y).
\]
Thus by setting $\psi(y):= \chi(y^0) G_x(y)$ and $v(y):=
\chi(y^0)u(y)$, we obtain
\begin{equation}\label{chi'}
\int_{\underline{t}}^{\overline{t}} dt\chi'(t) \left( u\sharpgdt G\right)(x)  = \int_{\R^{n+1}}dy
\left({\partial v\over \partial t}{\partial G_x\over \partial t} - {\partial u\over \partial t}{\partial \psi\over \partial t}\right)(y)
\end{equation}
Now since $G_x\in\mathcal{E}_0^{-s+1}$, i.e. $G_x$ is a weak solution of
(\ref{kg}), we have
\[
\int_{\R^{n+1}}dy {\partial v\over \partial t}{\partial G_x\over \partial t} = \int_{\R^{n+1}}dy \left(\vec{\nabla}v\cdot \vec{\nabla}G_x + m^2 vG_x\right) = \int_{\R^{n+1}}dy \chi \left(\vec{\nabla}u\cdot \vec{\nabla}G_x + m^2 uG_x\right).
\]
Similarly by using (\ref{weaksol}) and the fact that $\vec{\nabla}\chi = 0$, we deduce
\[
\int_{\R^{n+1}}dy {\partial u\over \partial t}{\partial \psi\over \partial t} = \int_{\R^{n+1}}dy \left(\vec{\nabla}u\cdot \vec{\nabla}\psi + m^2 u\psi + J\psi\right) = \int_{\R^{n+1}}dy \chi \left(\vec{\nabla}u\cdot \vec{\nabla}G_x + m^2 uG_x +  JG_x\right).
\]
Hence we deduce from (\ref{chi'}) that
\[
\int_{\underline{t}}^{\overline{t}} dt\chi'(t) \left( u\sharpgdt G\right)(x) = -\int_{\R^{n+1}}dy \chi(y^0)J(y) G(y-x) = -\int_{\underline{t}}^{\overline{t}} dt \chi(t)\int_ {\R^n} d\vec{y} J(y) G(y-x).
\]
Hence (\ref{Jsharp}) follows. Next, since $\left(J|_{t}\right)\sharp_tG\in \mathcal{E}^{s+1}_0$ we deduce from Lemma \ref{Ldeltapsi} that $\int_tJ\phi^\td$ defined by (\ref{Jphi}) is a bounded operator from $\F^{(1)}$ to $\F$ and that relation (\ref{Jsharp}) implies (\ref{delta/deltaphi}).
\end{proof}

\noindent Let us focus now on the definition of $\D(t) = -\int_tW(\phi^\tx,d\phi^\tx)\phi^\td$. We will actually consider a more general situation and consider, for $s\in\R$ an analytical function $\mathcal{V}$
\[
\mathcal{V}:\Omega\subset H^{s+1}(\R^n)\times H^s(\R^n)\longrightarrow H^s(\R^n)
\]
on a neighbourhood $\Omega\subset H^{s+1}(\R^n)\times H^s(\R^n)$ of the origin. More precisely we assume that there exists a family $(\mathcal{V}_p)_{p\in
\N}$ of continuous linear maps $\mathcal{V}_p: \bigotimes^p
\left(H^{s+1}(\R^n)\times H^s(\R^n)\right) \longmapsto H^s(\R^n)$.
We suppose that each $\mathcal{V}_p$ is symmetric, i.e. satisfies
(\ref{symmetric}). We denote by $\lc\mathcal{V}_p\rf$ the usual norm of $\mathcal{V}_p$,
i.e. the optimal constant in the inequality $||\mathcal{V}_p(\Phi_1\otimes\cdots \otimes \Phi_p)||_{H^s} \leq \lc\mathcal{V}_p\rf ||\Phi_1||_{H^{s+1}\times H^s}\cdots||\Phi_p||_{H^{s+1}\times H^s}$ (where, for $\Phi = (\phi_0,\phi_1)\in H^{s+1}(\R^n)\times H^s(\R^n)$, $||\Phi||_{H^{s+1}\times H^s}:= \left( ||\phi_0||_{H^{s+1}}^2 + ||\phi_1||_{H^s}^2\right)^{1/2}$). We suppose further that the power series $\lc \mathcal{V}\rf $ defined
by
\[
\lc \mathcal{V}\rf (z):=\sum_{q\ge 0}\lc\mathcal{V}_q\rf z^q
\]
has a positive radius of convergence $r_0>0$. Then $\forall \Phi\in H^{s+1}(\R^n)\times H^s(\R^n)$ such that $||\Phi||_{H^{s+1}\times H^s}<r_0$ the power series  $\mathcal{V}(\Phi):=\sum_{p= 0}^\infty \mathcal{V}_p(\Phi^{\otimes p})$ converges with respect to the $H^s(\R^n)$ topology and we have
\[
    ||{\cal V}(\Phi)||_{H^s}\le \lc\mathcal{V}\rf \left(||\Phi||_{H^{s+1}\times H^s}\right).
\]
\begin{rema}\label{remark2.1}
For $s>n/2$, since $H^s(\R^n)$ is a Banach algebra (see  Lemma \ref{lemmapp2}), the previous hypothesis is satisfied
by any functional $\mathcal{V}$ such that there exists a polynomial function or an analytic function $W:\omega\subset\R^{n+2}\longrightarrow\R$ on an open neighborhood of $0$ in $\R^{n+1}$ such that $\forall \Phi = (\phi_0,\phi_1)\in H^{s+1}(\R^n)\times H^s(\R^n)$, $\mathcal{V}(\Phi) = W(\phi_0,{\partial \phi_0\over \partial x^1},\cdots , {\partial \phi_0\over \partial x^n}, \phi_1)$.
\end{rema}

\noindent Given such a functional $\mathcal{V}$, for all $t\in\R$ and $\varphi\in {\cal E}^{s+1}_0$ such that $||\varphi||_{{\cal E}^{s+1}_0}<r_0$ the function $\mathcal{V}([\varphi]_t)$ is well defined and belongs to $H^s(\R^n)$. So by using Lemma \ref{lemmapp3} we deduce that $\mathcal{V}([\varphi]_t)\sharp_t G$ belongs to ${\cal E}^{s+1}_0$ and we can define
\[
\int_t \mathcal{V}(\source)\but := {\delta \over \delta (\mathcal{V}(\source)\sharp_t
G)},
\]
which means that, for all for all $r\in(0,r_0)$ and $|f\rangle\in \F_r$, $\left(\int_t \mathcal{V}(\source)\but\right)|f\rangle$ is defined by
\[
\forall \varphi\in B_{{\cal E}^{s+1}_0}(0,r), \quad
\left(\int_t \mathcal{V}(\source)\but\right)|f\rangle(\varphi) = {\delta |f\rangle \over \delta\left(\mathcal{V}([\varphi]_t)\sharp_t G\right)}(\varphi).
\]
\begin{lemm}
The operator $\int_t \mathcal{V}(\source)\but$ is a continuous operator from $\F^{(1)}_{(0,r_0)}$ to $\F_{(0,r_0)}$ and its norm is controlled by $\lc \mathcal{V}\rf (r)$
i.e.
\begin{equation}\label{estop3}
    \forall r\in(0,r_0), \quad \forall |f\rangle\in\F^{(1)}_r, \quad
    N_r\left(\int_t \mathcal{V}(\source)\but|f\rangle\right)\le
    \lc \mathcal{V}\rf (r) N_r^{(1)}(|f\rangle).
\end{equation}
\end{lemm}
\begin{proof}
    Consider $r\in(0,r_0)$ and let $|f\rangle \in \F_{pol}$ and write $|f\rangle(\varphi) = \sum_{p=1}^N|f_p\rangle(\varphi^{\otimes p})$ and $|g\rangle:= \int_t\mathcal{V}(\phi^\tx)\phi^\td|f\rangle$. Then, $\forall \varphi\in \mathcal{E}_0^{s+1}$ such that $||\varphi||_{{\cal E}^{s+1}_0}<r$ we know that ${\cal V}([\varphi]_t)$ is well defined and
    \[
    \begin{array}{ccl}
    |g\rangle(\varphi) & = & \displaystyle {\delta |f\rangle\over \delta (\mathcal{V}([\varphi]_t\sharp_tG)}(\varphi) = \sum_{p=1}^N p|f_p\rangle\left( (\mathcal{V}([\varphi]_t)\sharp_tG)\otimes \varphi^{\otimes p-1}\right) \\
     & = & \displaystyle \sum_{p=1}^N \sum_{q=0}^\infty p|f_p\rangle\left( (\mathcal{V}_q([\varphi]_t^{\otimes q})\sharp_tG)\otimes \varphi^{\otimes p-1}\right) = \sum_{k=0}^\infty g_k(\varphi^{\otimes k}),
    \end{array}
    \]
    where we have set $k=q+p-1$ and
    \[
    g_k(\varphi^{\otimes k}):= \sum_{p=1}^{\sup (N,k+1)} p|f_p\rangle\left( \left(\mathcal{V}_{k-p+1}([\varphi]_t^{\otimes k-p+1})\sharp_tG\right)\otimes \varphi^{\otimes p-1}\right).
    \]
    However
    \begin{equation}\label{estop4}
    |g_k(\varphi^{\otimes k})|\leq \sum_{p=1}^{\sup (N,k+1)} p \lc f_p\rf  ||\mathcal{V}_{k-p+1}([\varphi]_t^{\otimes k-p+1})\sharp_t G||_{\mathcal{E}^{s+1}} ||\varphi||_{\mathcal{E}^{s+1}}^{p-1}
    \end{equation}
    and since, by using Lemma \ref{lemmapp3},
    \[
    \begin{array}{ccl}
    ||\mathcal{V}_{k-p+1}([\varphi]_t^{\otimes k-p+1})\sharp_t G||_{\mathcal{E}^{s+1}} & \leq & \displaystyle  ||\mathcal{V}_{k-p+1}([\varphi]_t^{\otimes k-p+1})||_{H^s}\\
    & \leq & \displaystyle   \lc\mathcal{V}_{k-p+1}\rf ||[\varphi]_t||_{H^{s+1}\times H^s}^{k-p+1}\\
     & = & \displaystyle  \lc\mathcal{V}_{k-p+1}\rf ||\varphi||_{\mathcal{E}^{s+1}}^{k-p+1},
     \end{array}
    \]
    we deduce from (\ref{estop4}) that
    \[
    |g_k(\varphi^{\otimes k})|\leq \sum_{p=1}^{\sup (N,k+1)} p \lc f_p\rf \lc\mathcal{V}_{k-p+1}\rf ||\varphi||_{\mathcal{E}^{s+1}}^k.
    \]
    Hence, by posing $q=k-p+1$,
    \[
    \begin{array}{ccl}
    N_r(|g\rangle) & = & \displaystyle \sum_{k=0}^\infty \lc g_k\rf r^k \leq \sum_{k=0}^\infty \sum_{p=1}^{\sup (N,k+1)} p \lc f_p\rf \lc\mathcal{V}_{k-p+1}\rf r^k\\
     & = & \displaystyle \sum_{q=0}^\infty \sum_{p=1}^N p\lc f_p\rf \lc\mathcal{V}_q\rf r^qr^{p-1} = \displaystyle  \lc \mathcal{V}\rf (r) N_r^{(1)}(|f\rangle).
    \end{array}
    \]
    Thus we obtain (\ref{estop3}) for $|f\rangle\in \F_{pol}$. It implies the result by using the density of $\F_{pol}$ in $\F_r^{(1)}$.
\end{proof}

\section{Time ordered and ordinary exponentials of operators}\label{time_ordered_exponential_of_operators} 
The goal of this section is, given $t>0$, to define rigorously the time ordered exponential $T\hbox{exp}\left(\int_0^tds\D(s)\right)$ for a family $(\D(s))_{s\in [0,t]}$ of continuous operators $\D(s)$ from $\F^{(1)}_{(0,r_0)}$ to $\F_{(0,r_0)}$. Our motivation is to use later on this construction with $\D(s) = - \int_s \mathcal{V}\left(\phi^\tx, {\partial \phi^\tx\over \partial t}\right)\phi^\td$. We introduce the following notation:
\[
\hbox{for all }t\ge 0,\quad
\intit_0^t ds_1\cdots ds_k:=\int_{0>s_1>\cdots >s_k>t} ds_1 \cdots ds_k
\]
and we write
\begin{equation}\label{defTexp}
T \exp\left( \int_0^t d s \D(s)\right)
:= \sum_{k\ge 0}\intit_0^t ds_1\cdots ds_k \D(s_1) \D(s_2) \cdots \D(s_k).
\end{equation}

\subsection{The main result}\label{sub:the_main_result} 
\begin{theo}\label{theo.produit.ordonne.general}\sl{
Let $(\D(s))_{s\in [0,t]}$ be a family of continuous operators from $\F^{(1)}_{(0,r_0)}$ to $\F_{(0,r_0)}$ such that for any $r\in (0,r_0)$ and any $|f\rangle\in \F^{(1)}_r$, $[s\longmapsto \D(s) |f\rangle]$ belongs to $L^1([0,t],\F_r)$. Assume that there exists an analytic function $X(z)$ on the disc $B_\C(0,r_0)$ of the complex plane which satisfies ${d^k X \over dz^k}|_{z=0}\ge 0$ for all  $k\in \N$ and which controls the norm of $\D(s)$ for all $s\in[0,t]$. In other words assume that $\forall r\in(0,r_0)$,
\begin{equation}\label{xn1}
\forall s\in[0,t],\ \forall |f\rangle \in {\F_r^{(1)}},\quad
N_r(\D(s) |f\rangle)\le X(r)N_r^{(1)}(|f\rangle),
\end{equation}
Let $(t,z)\longmapsto e^{-tX}(z) := \gamma(t,z)$ be the solution of
\[
\left\{ \begin{array}{ccl}
\displaystyle{\partial \gamma\over \partial t}(t,z) & = & \displaystyle{-X(\gamma(t,z))}\\
\gamma(0,z) & = & z.
\end{array}\right.
\]
Then for all $t\ge 0$ and $r\in(0,r_0)$ such that $e^{-tX}(r)>0$, the operator $T \exp\left( \int_0^t d s \D(s)\right)$
defined by (\ref{defTexp}) is a bounded operator from $\F_r$ to $\F_{e^{-tX}(r)}$ with norm
less than $1$ i.e.
\begin{equation}\label{maininequa}
\forall |f\rangle\in  \F_r, \quad N_{e^{-tX}(r)}\left\{ T
\exp\left( \int_0^t d s \D(s)\right)|f\rangle\right\} \leq
N_r(|f\rangle).
\end{equation}
Moreover the map $[0,t]\longrightarrow \F_{e^{-tX}(r)}$,
$\tau\longmapsto T \exp\left( \int_0^\tau d s \D(s)\right)$ is
continuous. }
\end{theo}


\begin{proof}(of theorem \ref{theo.produit.ordonne.general})\\
Since $\F_{pol}$ is dense in $\F_r$ it suffices to show the result of $|f\rangle\in \F_{pol}$, then the result for $|f\rangle\in \F_r$ will be a consequence of (\ref{maininequa}). The proof is divided in several steps which are proved in the following.\\
\noindent
\emph{Step 1} --- We estimate the norm in $\F_r$ of
\[
\intit^t_0 ds_1\cdots ds_k\ \D(s_1)\D(s_2)\cdots \D(s_k)|f\rangle,
\quad \hbox{if }|f\rangle\in \F_{pol}.
\]
\begin{lemm}\label{reiterer}\sl{
Let $r>0$ and $k\in \N^*$ and assume that, for each $a = 1,\cdots ,k$, there exists a linear operator $\D_a:\F_{pol}\longrightarrow \F_r$ and an analytic vector field on $B_\C(0,r_0)$ with $r_0>r$
\[
X_a = X_a(z){d\over dz}:= \sum_{p=0}^\infty X_{a,p}z^p{d\over dz}, \quad \hbox{ for }z\in \C
\]
(so the power series $\sum_{p=0}^\infty X_{a,p}z^p$ converges for all $z\in \C$ s.t. $|z|<r_0$), such that $X_{a,p} \geq 0$, $\forall p\in \N$ and
\begin{equation}\label{hypolemma2.3}
\forall |f\rangle\in \F_{pol},\quad N_r(\D_a|f\rangle) \leq  X_a\cdot N_z(|f\rangle)|_{z=r} = X_a(r)N_r^{(1)}(|f\rangle).
\end{equation}
Then
\begin{equation}\label{NDDF}
\forall |f\rangle\in \F_{pol},\quad
N_r(\D_1 \cdots \D_k  |f\rangle) \leq \left[X_1 \cdots  X_k\cdot  N_z(|f\rangle)\right]|_{z=r}.
\end{equation}
}
\end{lemm}

\begin{proof}(of lemma \ref{reiterer})\\
We first observe that, by linearity, it suffices to prove the (\ref{NDDF}) for an arbitrary monomial functional $|f\rangle\in\F_{pol}$ of degree $j\in \N$, i.e. a functional of the type
$|f_j\rangle(\varphi)= |f_j\rangle(\varphi^{\otimes j})$ where
$|f_j\rangle \in S(\bigotimes^j\mathcal{E}^{s+1}_0)^*$. Then
\[
\forall \varphi\in {\cal E}^{q+1}_0\hbox{ s.t. }||\varphi||_{{\cal E}^{s+1}_0}<r;\quad
\D_a |f_j\rangle(\varphi) = \sum_{p=0}^\infty (\D_a |f_j\rangle )_{j+p-1}(\varphi^{\otimes j+p-1}),
\]
where, $\forall p\in\N$, $(\D_a |f_j\rangle)_{j+p-1}\in S(\bigotimes^{j+p-1} \mathcal{E}^{s+1}_0)^*$ and (\ref{hypolemma2.3}) reads
\begin{equation}\label{jXF}
\lc (\D_a |f_j\rangle)_{j+p-1}\rf \leq jX_{a,p} \lc f_j\rf ,\quad
\hbox{for }X_{a,p} \geq 0.
\end{equation}
However, still by linearity, we can further reduce the proof of (\ref{jXF}) $\Longrightarrow$ (\ref{NDDF}) for $|f\rangle = |f_j\rangle$ to the case where, $\forall a = 1,\cdots ,k$, there exists some $p_a\in \N$ such that $X_a(z) = X_{a,p_a}z^{p_a}{d\over dz} = \lambda_az^{p_a}{d\over dz}$, with $\lambda_a\geq 0$, i.e. $\D_a |f_j\rangle: \varphi\longmapsto (\D_a |f_j\rangle)_{j+p_a-1}(\varphi^{\otimes j+p_a-1})$, with the estimate $\lc(\D_a |f_j\rangle)_{j+p_a-1}\rf \leq j\lambda_a \lc f_j\rf $. Then, by writing $p_a':= p_a-1$,
\[
\left(\D_1 \cdots \D_k  |f_j\rangle\right) (\varphi) = \left(\D_1 \cdots \D_k  |f_j\rangle\right)_{j+p_1'+ \cdots + p_k'}(\varphi^{\otimes j+p_1' + \cdots + p_k'}),
\]
for some $\left(\D_1 \cdots \D_k |f_j\rangle\right)_{j+p_1'+ \cdots + p_k'}\in S(\bigotimes^{j+p_1'+ \cdots + p_k'}\mathcal{E}^{s+1}_0)^*$. On the one hand one proves recursively that
\[
\begin{array}{ccl}
\lc\left(\D_1 \cdots \D_k |f_j\rangle\right)_{j+p_1'+ \cdots + p_k'}\rf & \leq & \lambda_1(j+p_2'+\cdots +p_k') \lc\left(\D_2\cdots \D_k |f_j\rangle\right)_{j+p_2'+ \cdots + p_k'}\rf\\
& \leq & \cdots \\
& \leq & \lambda_1\cdots \lambda_k (j+p_2'+\cdots +p_k')\cdots (j+p_k')j \lc f_j\rf .
\end{array}
\]
On the other hand we have also
\[
\begin{array}{ccl}
\left(\lambda_1z^{p_1}{d\over dz}\right) \cdots \left(\lambda_kz^{p_k}{d\over dz}\right) N_z(|f_j\rangle) & = & \lc f_j\rf \left(\lambda_1z^{p_1}{d\over dz}\right) \cdots \left(\lambda_kz^{p_k}{d\over dz}\right)z^j\\
& = & \lambda_1\cdots \lambda_k (j+p_2'+\cdots +p_k')\cdots (j+p_k')j z^{j+p_1'+\cdots +p_k'}.
\end{array}
\]
Hence Inequality (\ref{NDDF}) follows directly from a comparison of the two results.
\end{proof}

\noindent We consider the holomorphic vector field on the disc $B_\C(0,r_0)$ of the complex plane defined by $X(z){d \over dz}$ which we denote by $X$. Then condition \refeq{xn1} of the theorem reads
\[
\forall s\in[0,t],\ \forall |f\rangle \in {\F_r^{(1)}},\quad
N_r(\D(s) |f\rangle)\le X\cdot N_z(|f\rangle)|_{z=r}.
\]
We can now apply the previous lemma to operator $\D(s)$ and, by using $\intit_0^t ds_1\cdots ds_k={t^k \over k!}$, we complete our first step with the:
\begin{coro} Assume that the hypotheses of theorem \ref{theo.produit.ordonne.general} are satisfied. Then $\forall |f\rangle \in \F_{pol}$,
\begin{equation}\label{Nek}
\forall r>0,\quad N_r\left( \intit^t_0 ds_1\cdots ds_k\ \D(s_1)\D(s_2)\cdots \D(s_k)|f\rangle \right)\le {t^k \over k!}X^k\cdot N_z(|f\rangle)|_{z=r}.
\end{equation}
\end{coro}
\begin{proof}
    Indeed Inequality (\ref{Nek}) is a straightforward consequence of Lemma \ref{reiterer} with $\D_a = \D(s_a)$ and $X_a = X$. (Note that (\ref{Nek}) can then extended to $|f\rangle\in \F_r^{(k)}$ by density of $\F_{pol}$ in $\F_r^{(k)}$.)
\end{proof}

\noindent
\emph{Step 2} --- We prove some results on the vector field
$X$ on the complex plane. In the following, for $T,R>0$, we denote by $\overline{B}(0,T):= \{\tau\in \C|\ |\tau|\leq T\}$ and
$\overline{B}(0,R):= \{z\in \C|\ |z|\leq R\}$.

\begin{lemm}\label{champdevecteur}\sl{
Let $X:B_\C(0,r_0)\longmapsto \C$ be an holomorphic vector field different from 0. Assume that
\[
X(z) = \sum_{k=0}^\infty X_kz^k,\quad \hbox{where }X_k\geq 0, \forall k\in \N.
\]
Let $R\in(0,r_0)$ and $T>0$ such that $e^{TX}(R)$ exists. Then the flow map
\[
\begin{array}{ccc}
\overline{B}(0,T)\times \overline{B}(0,R) & \longrightarrow & \C\\
(\tau,z) & \longmapsto & e^{\tau X}(z)
\end{array}
\]
is defined on $\overline{B}(0,T)\times \overline{B}(0,R)$ and in particular
\begin{equation}\label{controleflot}
\forall (\tau,z)\in \overline{B}(0,T)\times \overline{B}(0,R),\quad
|e^{\tau X}(z)| \leq e^{|\tau|X}(|z|)\leq e^{TX}(R).
\end{equation}
}
\end{lemm}
\begin{proof}(of lemma \ref{champdevecteur})\\
We will first show that $(\tau,z) \longmapsto e^{\tau X}(z)$ is defined and satisfies (\ref{controleflot}) over $B(0,T)\times B(0,R)$, where $B(0,T):= \{\tau \in \C|\ |\tau|<T\}$ and $B(0,R):= \{z\in \C|\ |z|<R\}$. Fix some $z\in B(0,R)$ and $\tau\in B(0,T)$. Then $\exists \varepsilon_0>0$ s.t. $\forall \varepsilon\in (0,\varepsilon_0]$,
\[
|z| \leq R - \varepsilon \quad \hbox{and}\quad
|t|\leq T_\varepsilon:= {T\over 1+\varepsilon}.
\]
We also let $\lambda\in S^1\subset \C$ such that $\tau = |\tau|\lambda$, where $0< |\tau| \leq T_\varepsilon$. We introduce the notations:
\[
\left\{ \begin{array}{ccll}
f_\varepsilon(t) &:= & e^{t(1+\varepsilon)X}(|z|+\varepsilon) & \forall t\in [0,T_\varepsilon]\\
\gamma(t) &:= & e^{t\lambda X}(z) & \forall t\in [0, \overline{t}) \\
g(t) & := & |\gamma(t)| & \forall t\in [0, \overline{t})
\end{array}\right.
\]
where $\overline{t}$ is the positive maximal existence time for $\gamma$. Note that $f_\varepsilon$ is defined on $[0,T_\varepsilon]$ because of the assumption that $e^{TX}(R)$ exists. Our first task is to show that the set:
\[
A_\varepsilon:= \{t\in [0,T_\varepsilon]\cap [0,\overline{t}) |\ g(t) - f_\varepsilon(t) \geq 0\}
\]
is actually empty. Let us prove it by contradiction and assume that
$A_\varepsilon\neq \emptyset$. Then there exists $t_0:= \inf
A_\varepsilon$. Note that $g(0) - f_\varepsilon(0) = - \varepsilon
<0$, hence we deduce from the continuity of $g- f_\varepsilon$ that
$t_0\neq 0$ and $g(t_0) = f_\varepsilon(t_0)$. Moreover since
$f_\varepsilon(0) = \varepsilon$ and $f_\varepsilon$ is increasing
because $X(r)>0$ for $r>0$ we certainly have $g(t_0) =
f_\varepsilon(t_0) >0$. We now observe that
\[
\forall z\in \C^*,\quad {\langle \lambda X(z),z\rangle \over |z|} = \left\langle \lambda \sum_{k=0}^\infty X_kz^k,{z\over |z|}\right\rangle \leq \sum_{k=0}^\infty X_k |z|^k = X(|z|).
\]
Hence for all $t\geq 0$ s.t. $g(t)\neq 0$,
\[
g'(t) = {\langle \lambda X(\gamma(t)),\gamma(t)\rangle \over |\gamma(t)|} \leq X(|\gamma(t)|) = X(g(t))
\]
and hence in particular, since $g(t_0) \neq 0$,
\[
g'(t_0) \leq X(g(t_0)) = X(f_\varepsilon(t_0)) = {f_\varepsilon'(t_0) \over 1+\varepsilon} < f_\varepsilon'(t_0).
\]
Thus since $f'_\varepsilon - g' $ is continuous $\exists t_1\in (0,t_0)$ s.t. $\forall t\in [t_1,t_0]$, $f'_\varepsilon(t) - g'(t)\geq 0$. Integrating this inequality over $[t_1,t_0]$ we obtain
\[
g(t_1) - f_\varepsilon(t_1) = \left( f_\varepsilon(t_0) - g(t_0)\right) -  \left( f_\varepsilon(t_1) - g(t_1)\right) = \int_{t_1}^{t_0} \left(f'_\varepsilon(t) - g'(t)\right) dt \geq 0,
\]
i.e. $t_1\in A_\varepsilon$, a contradiction. \\

\noindent
Hence $A_\varepsilon = \emptyset$. Note that this implies automatically that $\overline{t}> T_\varepsilon$. Indeed if we had $\overline{t}\leq T_\varepsilon$ this would imply that $g$ is not bounded in $[0,\overline{t}) \subset [0,T_\varepsilon]$, but since $f_\varepsilon$ is bounded on $[0,T_\varepsilon]$ we could then find some time $t\in [0,\overline{t})$ s.t. $g(t)\geq f_\varepsilon(t)$, which would contradict the fact that $A_\varepsilon = \emptyset$. Thus we deduce that $\forall t\in [0,T_\varepsilon]$, $g(t) < f_\varepsilon(t)$, i.e.
\[
\forall t\in [0,T_\varepsilon],\quad |e^{\lambda tX}(z)| < e^{(1+\varepsilon)tX}(|z|+\varepsilon).
\]
In other words for all $\tau = \lambda t\in B(0,T)$ and all $z\in B(0,R)$ we found that $\forall \varepsilon\in (0,\varepsilon_0]$, $|e^{\tau X}(z)|\leq e^{(1+\varepsilon)|\tau|X}(|z|+\varepsilon)$. Letting $\varepsilon$ goes to 0, we deduce the estimate (\ref{controleflot}) for $(\tau,z)\in B(0,T)\times B(0,R)$. Lastly this estimate forbids the flow to blow up on $\overline{B}(0,T)\times \overline{B}(0,R)$. Hence the result and (\ref{controleflot}) can be extended to this domain by continuity.
\end{proof}

\noindent We now consider any $r\in (0,r_0)$ and we let $\theta(r)$ to be the maximal positive time of existence for
$t \longmapsto e^{-tX}(r)$. (We remark that, if $X(0) =0$, then $\theta(r) = +\infty$, $\forall r\geq 0$.)
Then we have obviously that $\forall t\in [0,\theta(r))$, $e^{tX}\left((e^{-tX}(r)\right)$ exists since it is nothing but $r$.
Hence we can apply Lemma \ref{champdevecteur} with $(T,R) = (t,e^{-tX}(r))$. It implies that, for any
holomorphic function $h$ on the closed disc $\overline{B}(0,r)$ of the complex plane, the map
\[
\begin{array}{ccc}
\overline{B}(0,t)\times \overline{B}(0,e^{-tX}(r)) & \longrightarrow & \C\\
(\tau,z) & \longmapsto & h\left( e^{\tau X}(z) \right)
\end{array}
\]
is well defined and is analytic. Hence the following expansion holds:
\begin{equation}\label{taylor}
\forall (\tau,z) \in \overline{B}(0,t)\times \overline{B}(0,e^{-tX}(r)),\quad
h\left( e^{\tau X}(z) \right) =
\sum_{k=0}^\infty {d^k\over (ds)^k} \left[ h\left( e^{s X}(z) \right)\right]|_{s=0} {\tau^k\over k!}.
\end{equation}
However because of the following identity
\[
{d^k\over (ds)^k} \left[ h\left( e^{sX}(z) \right)\right] = (X^k\cdot h)(e^{sX}(z))
\]
(which can be proved by recursion over $k$), we deduce from (\ref{taylor}) that
\[
\forall (\tau,z) \in \overline{B}(0,t)\times \overline{B}(0,e^{-tX}(r)),\quad
h\left( e^{\tau X}(z) \right) = \sum_{k=0}^\infty (X^k\cdot h)(z) {\tau^k\over k!}.
\]
By specializing this relation to $(\tau,z) = (t,e^{-tX}(r))$ we deduce that the power series $\sum_{k=0}^\infty {t^k\over k!}(X^k\cdot h)(e^{-tX}(r))$ is absolutely convergent and satisfies the identity
\begin{equation}\label{identity=}
h(r) = \sum_{k=0}^\infty {t^k\over k!}(X^k\cdot h)(e^{-tX}(r)).
\end{equation}

\noindent
\emph{Step 3} --- We complete the proof of the Theorem.
We prove the estimate (\ref{maininequa}). By first applying (\ref{Nek}) for some $|f\rangle\in \F_{pol}$, we have
\[
\sum_{k=0}^\infty N_{e^{-tX}(r)}\left[ \intit^t_0 ds_1\cdots ds_k\ \D(s_1)\D(s_2)\cdots \D(s_k)|f\rangle \right]
 \leq  \sum_{k=0}^\infty {t^k\over k!} X^k\cdot N_z(|f\rangle)|_{z=e^{-tX}(r)}.
\]
But using then (\ref{identity=}) with $h(z) = N_z(|f\rangle)$ we obtain
\[
\forall |f\rangle\in \F_{pol}, \quad
\sum_{k=0}^\infty {t^k\over k!} (X^k\cdot N_z(|f\rangle))|_{z=e^{-tX}(r)} = N_r(|f\rangle).
\]
Hence the series $T \hbox{exp}\left( \int_0^t ds\D(s)\right) |f\rangle$ converges in $\F_{e^{-tX}(r)}$ and we have the estimate
\[
\begin{array}{ccl}
\displaystyle N_{e^{-tX}(r)}\left[ T \hbox{exp}\left( \int_0^t ds\D(s)\right) |f\rangle \right] & \leq & \displaystyle  \sum_{k=0}^\infty  N_{e^{-tX}(r)}\left[ \intit^t_0 ds_1\cdots ds_k\ \D(s_1)\D(s_2)\cdots \D(s_k)|f\rangle \right]\\
& \leq & N_r(|f\rangle).
\end{array}
\]
So (\ref{maininequa}) follows by density of $\F_{pol}$ in $\F_r$. Let us prove the last part of the theorem. Let $\tau\in[0,t]$, then for all $s\in[0,t]$ we have
\[
T  e^{\int_0^\tau d\sigma \D(\sigma)}|f\rangle-
T  e^{\int_0^s d\sigma \D(\sigma)}|f\rangle
=\sum_{k\ge 0}\int_s^\tau d\sigma \D(\sigma) \intit_0^\sigma dt_1\cdots dt_k \D(t_1)\cdots \D(t_k)|f\rangle.
\]
Hence using the same computations as above we get
\[
\left|
T  e^{\int_0^\tau d\sigma \D(\sigma)}|f\rangle-
T  e^{\int_0^s d\sigma \D(\sigma)}|f\rangle
\right|
\le
|s-\tau| \sum_{k\ge 0}{t^k \over k!}X^{k+1}\cdot N_z(|f\rangle)|_{z=\gamma}
\]
which tends to $0$ when $s\to \tau$.
\end{proof}

As a direct consequence of theorem \ref{theo.produit.ordonne.general} a similar result holds for the ordinary exponential of a linear operator from $\F^{(1)}_{(0,r_0)}$ to $\F_{(0,r_0)}$ as follows.
\begin{coro}\label{thmbis}
Let $\T$ be a continuous operator from $\F^{(1)}_{(0,r_0)}$ to $\F_{(0,r_0)}$ with norm controlled by an holomorphic function $X$ on $B_\C(0,r_0)$ such that for all $k\ge 0$, ${d^kX \over dz^k}|_{z=0}\ge 0$. Then for all $r\in(0,r_0)$ and $\sigma >0$ such that $e^{-\sigma X}(r) >0$ the operator
\[
e^{\sigma \T}:= \sum_{k=0}^\infty {(\sigma \T)^k\over k!}: \F_r\longrightarrow \F_{e^{-\sigma X}(r)}
\]
is well-defined, continuous and satisfies
\[
\forall |f\rangle\in \F_r,\quad N_{e^{-\sigma X}(r)}\left(e^{\sigma \T}|f\rangle\right) \leq N_r(|f\rangle).
\]
\end{coro}
\begin{proof}
    Just take $\D(s)=\T$ for all $s\in[0,\sigma]$ in theorem \ref{theo.produit.ordonne.general}, then in this case the times order exponential reduces to the usual exponential since for all $k\in\N$ we have $\int_{0<s_1<\cdots<s_k<\sigma}ds_1\cdots ds_k={1 \over k!}\sigma^k$.
\end{proof}

\subsection{Applications of the main result}\label{sub:applications_of_the_main_result} 
\noindent Let $u$ belong to $\cap_{\ell=0,1}{\cal C}^\ell((\underline{t},\overline{t}),H^{s+1-l}(\R^n))$ then we can apply Corollary \ref{thmbis} to $\U(t):= \int_tu\gd \phi^\td$ and we obtain the:
\begin{coro}\label{coroexp}
Let $t\in(\underline{t},\overline{t})$ and  $\kappa:= ||u|_t||_{H^{s+1}}+||{\partial u \over \partial t}|_t||_{H^{s}} = ||[u]_t||_{s+1}$ then the following linear operator is well-defined and bounded
\[
e^{\U(t)}: \F_r \longrightarrow \F_{r-\kappa},\quad\hbox{ for }r>\kappa;
\]
and satisfies the estimate
\begin{equation}\label{estimexpU}
\forall |f\rangle\in \F_r,\quad
N_{r-\kappa}(e^{\U(t)}|f\rangle) \leq N_r(|f\rangle).
\end{equation}
Moreover, $\forall |f\rangle\in \F_r$, we have the identity: $\forall \varphi\in \mathcal{E}_0^{s+1}$ s.t. $||\varphi||_{{\cal E}^{s+1}_0}<r-\kappa$
\begin{equation}\label{demystifier}
\left(e^{\U(t)}|f\rangle\right)(\varphi) = |f\rangle \left( \varphi + u\sharpgdt G\right)
\end{equation}
\end{coro}
In particular if we take $\varphi=0$ in \refeq{demystifier} we get
\begin{equation}\label{demystifier2}
    \forall r>\kappa;\quad \forall |f\rangle\in\F_{r};\quad \langle 0 | e^{\U(t)}|f\rangle = |f\rangle\left(u\sharpgdt G\right),
\end{equation}
which is exactly $\langle [u]_t|f\rangle$ since $u\sharpgdt G$ is the element $\varphi_{[u]_t}$ of ${\cal E}^{s+1}_0$ such that $[\varphi_{[u]_t}]_t=[u]_t$. Hence we finally get
\[
\langle 0|e^{\U(t)}=\langle [u]_t|\in \bigcap_{r>\kappa}(\F_r)^*.
\]

\begin{proof}
Let $\U$ denote temporarily the operator $\U(t)$ (here $t$ is fixed). The existence of $e^{\U}$ and the estimate (\ref{estimexpU}) follow from Theorem \ref{thmbis} with respectively $\T = \U$ and $X(r) = \kappa$ (because of (\ref{estop1})). To prove the identities (\ref{demystifier}) for $\left(e^{\U}|f\rangle\right)(\varphi)$ let us consider $|f\rangle=|f_p\rangle\in S(\bigotimes^p\mathcal{E}_0^{s+1})^*$, then using the definition of $\U$ we have $\forall\varphi\in{\cal E}^{s+1}_0$
\begin{equation}\label{pf.demystifier}
    e^{\U}|f_p\rangle(\varphi)=\sum_{k=0}^p {p(p-1)\cdots(p-k+1) \over k!} |f_p\rangle ((u\sharpgdt G)^{\otimes k}\otimes \varphi^{\otimes(p-k)})
\end{equation}
which is exactly $|f_p\rangle\left((\varphi + u\sharpgdt G)^{\otimes p}\right)$ since $|f_p\rangle$ is symmetric. Then let $r>\kappa$, if one takes $\varphi\in{\cal E}_0^{s+1}$ such that $||\varphi||_{{\cal E}^{s+1}_0}<r-\kappa$, then using properties of $u\sharpgdt G$ we have $||\varphi + u\sharpgdt G||_{{\cal E}_0^{s+1}}< r$. So if $|f\rangle=\sum_p |f_p\rangle$ belongs to $\F_r$ then thanks to \refeq{pf.demystifier} we have
\[
e^{\U}|f\rangle=\sum_{p\ge 0} e^{\U}|f_p\rangle = \sum_p |f_p\rangle \left((\varphi + u\sharpgdt G)^{\otimes p}\right)
\]
which leads to \refeq{demystifier} since $||\varphi + u\sharpgdt G||_{{\cal E}_0^{s+1}}< r$ and $|f\rangle\in\F_r$.
\end{proof}
As a byproduct of this section we apply Theorem \ref{theo.produit.ordonne.general} to the family of operator $\D(s):=-\int_s {\cal V}(\source)\but$, ${s\in[0,t]}$ with $X(z) :=\lc \mathcal{V}\rf (z):=\sum_{q\ge 0} \lc {\cal V}_q \rf z^q$. Let $(t,z)\longmapsto e^{-tX}(z)$ denotes the solution of
    \[
    \left\{ \begin{array}{ccl}
    \displaystyle{\partial \gamma\over \partial t}(t,z) & = & \displaystyle{-\lc {\cal V} \rf (\gamma(t,z))}\\
    \gamma(0,z) & = & z.
    \end{array}\right.
    \]
    Then for all $r\in(0,r_0)$ such that $e^{-tX}(r)>0$ the time ordered exponential
    \[
        T\exp\left(-\int_0^t ds \D(s)\right):\F_r\longrightarrow\F_{e^{-tX}(r)}
    \]
    is a bounded operator with norm less than $1$. \\

\noindent \textbf{Conclusion} --- Now we can inspect in which circumstances it is possible to make sense of formula (\ref{Ftbasic}), i.e. to define
\begin{equation}\label{Flagrande}
\mathcal{F}_t^\varphi[u]_t:= \langle [u]_t| T\hbox{exp}\left(\int_0^tds\D(s)\right) \int_0\phi^\tx\gd  \varphi|0\rangle,
\end{equation}
where we recall that $\langle[u]_t|=\langle 0|e^{\U(t)}$.
For all $r\in(0,r_0)$ such that $e^{-tX}(r)>0$, $T\hbox{exp}\int_0^tds\D(s)$ maps $\F_r$ to $\F_{e^{-tX}(r)}$ and, if furthermore $e^{-tX}(r)-\kappa>0$ (where $\kappa = ||[u]_t||_{s+1}$ as in Corollary \ref{coroexp}), $e^{\U(t)}$ maps $\F_{e^{-tX}(r)}$ to $\F_{e^{-tX}(r)-\kappa}$. Hence we deduce that
\begin{equation}\label{composedoperator}
\hbox{if }e^{-tX}(r)>\kappa,\quad
e^{\U(t)}\circ T\hbox{exp}\left(\int_0^tds\D(s)\right) \hbox{ maps continuously }\F_r\hbox{ to }\F_{e^{-tX}(r)-\kappa}.
\end{equation}
Now recall that $\int_0\phi^\tx\gd  \varphi|0\rangle\in \F_{pol}\subset\cap_{r>0}\F_r$. Thus its image by the operator in (\ref{composedoperator}) makes sense and belongs to $\cap_{r\in(0,r_0)}\F_{e^{-tX}(r)-\kappa}$ if there exists $r\in(0,r_0)$ such that $e^{-tX}(r)>\kappa$. Then $\mathcal{F}_t^\varphi[u]_t$ is obtained by evaluating this image on $\langle 0|$. Hence \textbf{a necessary and sufficient condition to define $\mathcal{F}_t^\varphi[u]_t$ through (\ref{Flagrande}) is is that there exists some $r\in(0,r_0)$ such that $e^{-tX}(r)>\kappa$, which, since $e^{tX}$ is monotone increasing on $(0,\infty)$, is equivalent to the condition:}
\begin{equation}\label{expKappaFini}
 e^{tX}\left(\kappa\right) = e^{tX}\left(||[u]_t||_{s+1}\right) <\infty.
\end{equation}

\section{Conserved quantities}\label{conserved_quantities} 
In this section we show that $\mathcal{F}_t^\varphi[u]_t$ given by \refeq{magic} does not depends on $t$ if $u$ satisfies $(\Box+m^2)u+{\cal V}(u,{\partial u \over \partial t})=0$ in the distribution sense. Actually we prove a more general result: assuming some condition on $u$ and $r$ we show that for any $|f\rangle\in\F_r$ the quantity
\[
    \langle [u]_t|T\hbox{exp}\left(\int_0^tds\D(s)\right)|f\rangle
\]
does not depend on time. We recover $\mathcal{F}_t^\varphi[u]_t$ by taking $|f\rangle=|f_\varphi\rangle\in\F_{pol}$ defined by \refeq{magic2}. Theorem \ref{mainthm} follows by choosing $s>n/2$ and $\mathcal{V}$ to be a local functional (see Remark \ref{remark2.1}).
\begin{theo}
    Consider $u\in\mathcal{C}^0((\underline{t},\overline{t}),H^{s+1}(\R^n))\cap \mathcal{C}^1((\underline{t},\overline{t}),H^s(\R^n))$ and suppose that
    \begin{equation}\label{condition.cons}
        \sup_{t\in(\underline{t},\overline{t})} e^{tX}\left(||[u]_t||_{s+1}\right)<r_0.
    \end{equation}
    Let $r\in(0,r_0)$ be such that $\sup_{t\in(\underline{t},\overline{t})} e^{tX}(||[u]_t||_{s+1})<r$, then for all $|f\rangle\in\F_{r}$ the quantity
    \begin{equation}\label{alpha(t)}
        \alpha(t):=\langle 0|\exp\left(\int_t u\gd\but\right)T\exp\left(-\int_0^t ds\int_s {\cal V}\left(\source,{\partial \source\over \partial t}\right)\but\right)|f\rangle
    \end{equation}
    is well defined. Moreover if $u$ satisfies $(\Box+m^2)u+{\cal V}(u,{\partial u \over \partial t})=0$ in the distribution sense, then $\alpha(t)$ does not depend on $t\in(\underline{t},\overline{t})$.
\end{theo}
\begin{proof}
    Let $|f\rangle$ belong to $\F_r$ and consider the function $\alpha:(\underline{t},\overline{t})\longrightarrow\R$  defined by (\ref{alpha(t)}). Since we have supposed condition (\ref{condition.cons}) the results of the previous section imply that $\alpha$ is well defined. Now we will show that, if $\exists J\in \mathcal{C}^0((\underline{t},\overline{t}),H^{s+1}(\R^n))$ such that $\square u+m^2 u + J = 0$ in the distribution sense, $\alpha$ admits a derivative and $\forall t\in(\underline{t},\overline{t})$
    \begin{equation}\label{base.dem}
    \alpha'(t) = \langle 0|\int_t\left[(J-{\cal V}([u]_t)\right]\phi^\td\exp\left(\int_t u\gd\but\right)T\exp\left(-\int_0^t ds\int_s {\cal V}\left(\source,{\partial \source\over \partial t}\right)\but\right)|f\rangle.
\end{equation}
    Hence if $J = {\cal V}([u]_t)$ we deduce $\alpha'(t)=0$ which completes the proof of the theorem.
    
    Let us recall some notations : for all $\tau\in\R$ we denote by $\D(\tau)$, $\U(\tau)$ and $\V(\tau)$ the operators
\[
    \D(\tau):=-\int_\tau {\cal V}\left(\source,{\partial \source\over \partial t}\right)\but ; \quad \U(\tau):=\int_\tau u\gd \but ;\quad
    \V(\tau):=\int_\tau J\but;\quad \langle [u]_\tau|:=\langle 0|e^{\U(\tau)}.
\]
Then thanks to Lemma \ref{lemme.u} we know that ${d \U(\tau) \over d\tau}=\V(\tau)$, and Corollary \ref{coroexp} shows that $\langle [u]_t|g\rangle =  |g\rangle(u\sharpgdt G)$. Using these notations $\alpha(t)$ just reads $\langle [u]_t| T\exp\left(\int_0^t ds\D(s)\right)|f\rangle$. \\

\noindent For all $t\in(\underline{t},\overline{t})$ and for all $r>0$ we denote by $R_{u,t}(r)$ the quantity
\[
    R_{u,t}(r):=r-||[u]_t||_{s+1},
\]
so if $r\in(0,r_0)$ satisfies $\sup_{t\in(\underline{t},\overline{t})} e^{tX}(||[u]_t||_{s+1})<r$ then we have $R_{u,t}(e^{-tX}(r))>0$ for all $t\in(\underline{t},\overline{t})$. Let fix $t\in (\underline{t},\overline{t})$ then using the continuity of $(t,r)\longmapsto R_{u,t}(r)$ and $e^{-tX}(r)$ we get :
\begin{itemize}
    \item  There exists $r'>0$ such that $0<r'<r$ and $R_{u,t}(e^{-tX}(r'))>0$.
    \item  Since $0<r'<r$ we have $e^{-tX}(r)>e^{-tX}(r')$ hence we can find $\varepsilon'_0>0$ such that for all $h\in\R$ s.t. $|h|<\varepsilon'_0$ we have $e^{-(t+h)X}(r)>e^{-tX}(r')$.
    \item We have $R_{u,t}(e^{-tX}(r'))>0$, so we can take $\gamma'>0$ satisfying $0<\gamma'<e^{-tX}(r')$ and $R_{u,t}(\gamma')>0$ and then there exists $\varepsilon''_0>0$ such that for all $h\in\R$ s.t. $|h|<\varepsilon''_0$ we have $R_{u,t+h}(\gamma')>0$.
\end{itemize}
Finally we set $\varepsilon_0=\min(\varepsilon'_0,\varepsilon''_0)$ and we have for all $|h|<\varepsilon_0$, $e^{-(t+h)X}(r)>e^{-tX}(r')>\gamma'$ and $R_{u,t+h}(\gamma')>0$ which ensures that for all $|h|<\varepsilon_0$
\[
  T\exp\left(\int_0^{t+h}ds \D(s)\right)|f\rangle\in\F_{e^{-(t+h)X}(r)}\subset \F_{e^{-tX}(r')}\subset \F_{\gamma'},
\]
and $\langle [u]_{t}|$ belongs to $\F_{\gamma'}^*$. So for all $|h|<\varepsilon_0$ we can decompose  $(\alpha(t+h)-\alpha(t))/h$ in the following way
\[
    {1 \over h}(\alpha(t+h)-\alpha(t))=(I)_h + (II)_h
\]
where $(I)_h$ and $(II)_h$ is defined by
\begin{align*}
    (I)_h:=&{1 \over h} \left[\langle [u]_{t+h}|-\langle [u]_{t}|\right] T\exp\left(\int_0^{t+h} ds\D(s)\right)|f\rangle \\
    (II)_h:=& \langle [u]_{t}|{1 \over h}\left[T\exp\left(\int_0^{t+h} ds\D(s)\right)-T\exp\left(\int_0^{t} ds\D(s)\right) \right]|f\rangle
\end{align*}
\begin{lemm}\label{derive.translation}
    Let $\gamma''>\gamma'$ then we have the following limit with respect to the $\F_{\gamma''}^*$ topology
    \[
        \lim_{h\to 0}{1 \over h}(\langle [u]_{t+h}|-\langle [u]_{t}|) = \langle [u]_{t}|\V(t).
    \]
\end{lemm}
\begin{proof}
    Let $|g\rangle$ belong to $\F_{\gamma'}$ and we write $|g\rangle:=\sum_{p\ge 0}|g_p\rangle$ where for all $p\in\N$, $|g_p\rangle\in S(\bigotimes^p{\cal E}_0^{s+1})^*$. We denote temporarily by $A(t)$ and $B(t)$ the functions defined by
    \[
        A(t):=u\sharpgdt G \quad \hbox{and} \quad B(t):=(J|_t)\sharp_t G.
    \]
    Then using identity \refeq{demystifier} we get for all $p\in\N^*$
    \[
        { \left(\langle [u]_{t+h}|-\langle [u]_{t}|\right)|g_p\rangle \over h}= 
        \sum_{l=0}^{p-1}|g_p\rangle\left({1 \over h}\left[A(t+h) -A(t)\right]\otimes A(t+h)^{\otimes l}\otimes A(t)^{\otimes(p-1-l)}\right).
    \]
    On the other hand we have $\langle [u]_{t}| \V(t) |g_p\rangle =  p|g_p\rangle\left(B(t) \otimes A(t)^{\otimes (p-1)}\right)$. So since $|g_p\rangle$ is symmetric we can write
    \begin{multline*}
        \langle [u]_{t}| \V(t) |g_p\rangle = \sum_{l=0}^{p-1}|g_p\rangle \left(B(t) \otimes A(t+h)^{\otimes l}\otimes A(t)^{\otimes(p-1-l)}\right) \\
        + \sum_{l=0}^{p-1}\sum_{k=0}^{l-1} |g_p\rangle (B(t) \otimes (A(t+h)-A(t))\otimes A(t+h)^{\otimes k}\otimes A(t)^{\otimes (p-2-k)}).
    \end{multline*}
    Then these two identities lead to the following estimation
    \begin{multline*}
        \left| \left[{\langle [u]_{t+h}|-\langle [u]_{t}| \over h}-\langle [u]_{t}|\V(t)\right]|g_p\rangle\right|\le 
        ||{\Delta(h) \over h}-B(t)||_{{\cal E}_0^{s+1}}\sum_{l=0}^{p-1} \lc g_p \rf
        ||A(t+h)||_{{\cal E}_0^{s+1}}^{l}||A(t)||_{{\cal E}_0^{s+1}}^{p-1-l}\\
        +||\Delta(h)||_{{\cal E}_0^{s+1}}||B(t)||_{{\cal E}_0^{s+1}}\sum_{l=0}^{p-1}\sum_{k=0}^{l-1}\lc g_p \rf ||A(t+h)||_{{\cal E}_0^{s+1}}^{k}||A(t)||_{{\cal E}_0^{s+1}}^{p-2-k}
    \end{multline*}
    where $\Delta(h)$ denotes the function $\Delta(h):=A(t+h)-A(t)$.
    But for all $|h|<\varepsilon_0$ we know that $R_{u,t+h}(\gamma')>0$ hence we have $||A(t+h)||_{{\cal E}_0^{s+1}}<\gamma'$. So by summing the previous estimation for all $p\in\N$ and using the fact that $|g\rangle\in\F_{\gamma'}$ we finally get
    \begin{multline*}
        \left| \left[{\langle [u]_{t+h}|-\langle [u]_{t}| \over h}-\langle [u]_{t}|\V(t)\right]|g\rangle\right|\le
        ||{\Delta(h) \over h}-B(t)||_{{\cal E}_0^{s+1}} N^{(1)}_{\gamma'}(|g\rangle) \\
        +{1 \over 2}||\Delta(h)||_{{\cal E}_0^{s+1}}||B(t)||_{{\cal E}_0^{s+1}} N^{(2)}_{\gamma'}(|g\rangle).
    \end{multline*}
    Consider $\gamma''>\gamma'$ then we have the continuous injection $\F^{(k)}_{\gamma'}\subset \F_{\gamma''}$ for all $k\in\N$. Hence there exists $\mu,\nu>0$ such that $N^{(1)}_{\gamma'}\le \mu N_{\gamma''}$ and $N^{(2)}_{\gamma'}\le \nu N_{\gamma''}$. Moreover since $\F_{\gamma'}$ is dense in $\F_{\gamma''}$ we finally get that for all $|h|<\varepsilon_0$
        \[
        \left|\left|\left[{\langle [u]_{t+h}|-\langle [u]_{t}| \over h} - \langle [u]_{t}|\V(t)\right]\right|\right|_{\F^*_{\gamma''}}\le
        \mu\left|\left|{\Delta(h) \over h}-B(t)\right|\right|_{{\cal E}_0^{s+1}}+ \nu' ||\Delta(h)||_{{\cal E}_0^{s+1}},
        \]
    where $\nu'$ denotes $\nu':=\nu||B(t)||_{{\cal E}_0^{s+1}}$. But $u\in
    \cap_{\ell=0,1,2}\mathcal{C}^\ell((\underline{t},\overline{t}),H^{s+2-\ell}(\R^n))$, so lemma \ref{lemmapp3} ensures that we have the following limits with respect to ${\cal E}_0^{s+1}$ topology :
    \[
            \lim_{h\to 0}{\Delta(h) \over h}=\lim_{h\to 0}{1 \over h}\left[A(t+h)-A(t)\right]=B(t),
    \]
    which completes the proof.
\end{proof}

\noindent Let us take $\gamma''=e^{-tX}(r')>\gamma'$ in lemma \ref{derive.translation}, then thanks to theorem \ref{theo.produit.ordonne.general} we know that $h\longmapsto T\exp\left(\int_0^{t+h}ds \D(s)\right)|f\rangle$ is continuous from $(t-\varepsilon_0,t+\varepsilon_0)$ to $F_{e^{-tX}(r')}$, so we get
\begin{equation}\label{lim.I}
    \lim_{h\to 0}(I)_h = \langle [u]_{t}| \V(t) T\exp\left(\int_0^{t}ds \D(s)\right)|f\rangle.
\end{equation}
\begin{lemm}\label{lemmDeriveD}
    Consider $\varphi\in{\cal E}_0^{s+1}$ be such that $||\varphi||_{{\cal E}_0^{s+1}}<\gamma'$ then $T\exp\left(\int_0^\tau ds\D(s)\right)|f\rangle(\varphi)$ admits derivative with respect to $\tau$ for $\tau=t$ and we have
    \[
        {d \over d\tau}\left.\left(T\exp\left(\int_0^\tau ds\D(s)\right)|f\rangle(\varphi)\right)\right|_{\tau=t}=\D(t)T\exp\left(\int_0^t ds\D(s)\right)|f\rangle(\varphi)
    \]
\end{lemm}
\begin{proof}
    Notice that we have chosen $\gamma'$ such that  $T\exp\left(\int_0^\tau ds\D(s)\right)|f\rangle$ belongs to $\F_{\gamma'}$ for all $\tau\in(t-\varepsilon_0,t+\varepsilon_0)$. Let $|v(\tau)\rangle$ denotes $|v(\tau)\rangle:=T\exp\left(\int_0^\tau ds\D(s)\right)|f\rangle\in\F_{\gamma'}$. Then using definition of the ordered exponential we have for all $|h|<\varepsilon_0$
    \[
        {1 \over h}(|v(t+h)\rangle-|v(t)\rangle)(\varphi)={1 \over h}\int_t^{t+h}d\tau \D(\tau) |v(\tau)\rangle(\varphi).
    \]
    Let us write $|v(\tau)\rangle=\sum_{p\ge 0}|v_p(\tau)\rangle$ where for all $p\in\N$, $|v_p(\tau)\rangle\in S(\bigotimes^p{\cal E}_0^{s+1})^*$. Then for all $p\in\N^*$ we have
    \[
        {1 \over h}\int_t^{t+h}d\tau \D(\tau) |v(\tau)\rangle(\varphi)={1 \over h}\int_t^{t+h}d\tau p|v_p(\tau)\rangle \left(\left({\cal V}([\varphi]_\tau)\sharp_\tau G\right)\otimes \varphi^{\otimes (p-1)}\right).
    \]
    But we know that $\tau\longmapsto[\varphi]_\tau\in H^{s+1}\times H^s$ and ${\cal V}$ are continuous so that the map $\tau\longmapsto{\cal V}([\varphi]_\tau)\in H^{s}$ is continuous. Finally lemma \ref{lemmapp3} and theorem  \ref{theo.produit.ordonne.general} gives respectively that $\tau\longmapsto{\cal V}([\varphi]_\tau)\sharp_\tau G\in {\cal E}_0^{s+1}$ and $\tau\longmapsto|v(\tau)\rangle$ are continous. So we finally get that for all $p\in\N^*$
    \[
        \lim_{h\to 0}{1 \over h}\int_t^{t+h}d\tau \D(\tau) |v(\tau)\rangle(\varphi)=\D(t)|v_p(t)\rangle.
    \]
    On the other hand, using properties of operator $\D(\tau)$ we have
    \[
        \left|p|v_p(\tau)\rangle \left(\left({\cal V}([\varphi]_\tau)\sharp_\tau G\right)\otimes \varphi^{\otimes (p-1)}\right)\right|
        \le  \lc {\cal V} \rf (\gamma') p \lc v_p(\tau) \rf ||\varphi||_{{\cal E}_0^{s+1}}^{p-1}
    \]
    so since $||\varphi||_{{\cal E}_0^{s+1}}<\gamma'$ and $|v_p(\tau)\rangle$ is continuous from $(t-\varepsilon_0,t+\varepsilon_0)$ to $\F_{\gamma'}$  we can apply Lebesgue theorem in order to conclude.
\end{proof}

\noindent Using identity \refeq{demystifier} we get the following expression for $(II)_h$
\[
    (II)_h={1 \over h}\left[T\exp\left(\int_0^{t+h} ds\D(s)\right)-T\exp\left(\int_0^{t} ds\D(s)\right) \right]|f\rangle (u\sharpgdt G).
\]
Thus since we have $||u\sharp_t G||_{{\cal E}_0^{s+1}}<\gamma'$ we can take $\varphi=u\sharp_t G$ in lemma \ref{lemmDeriveD} and we get
\[
    \lim_{h\to 0} (II)_h = \D(t)T\exp\left(\int_0^t ds\D(s)\right)|f\rangle(u\sharpgdt G).
\]
Let us denote by $|w\rangle=\sum_{p\ge 0}|w_p\rangle$ the element $|w\rangle:=T\exp\left(\int_0^t ds\D(s)\right)|f\rangle$ of $\F_{\gamma'}$, as usual we assume that for all $p\in\N$, $|w_p\rangle$ belongs to $S(\bigotimes^p{\cal E}_0^{s+1})^*$. Then using definition of $\D(t)$ the previous limit reads
\[
    \D(t)|w\rangle(u\sharpgdt G)= -\sum_{p\ge 0}p|w_p\rangle \left(\left[{\cal V}\left([u\sharpgdt G]_t\right)\right]\sharp_t G\otimes (u\sharpgdt G)^{\otimes (p-1)}\right).
\]
But directly from the definition of $G$ we see that $[u\sharpgdt G]_t=(u|_t,{\partial u \over \partial t}|_t)=:[u]_t$. So the right hand side of the previous identity reads
\[
    -\sum_{p\ge 0}p|w_p\rangle \left(\left[{\cal V}\left([u\sharpgdt G]_t\right)\right]\sharp_t G\otimes (u\sharpgdt G)^{\otimes (p-1)}\right) = -\int_t {\cal V}([u]_t)\but |w\rangle (u\sharpgdt G).
\]
So using  \refeq{demystifier}, we finally get
\[
\lim_{h\to 0} (II)_h = -\langle [u]_{t}| \left[\int_t {\cal V}([u]_t)\but\right] T\exp\left(\int_0^t ds\D(s)\right)|f\rangle
\]
which together with \refeq{lim.I} and lemma \ref{lemme.u} leads to \refeq{base.dem}.
\end{proof}

\section{Generalization to systems of PDEs}\label{sec:generalization_to_systems} 
The previous construction can be adapted without difficulty for systems of PDE. More precisely, given $q\in\N^*$ we consider a system of PDE which reads
\begin{equation}\label{kg.system}
    {\partial^2 u \over \partial t^2}-\Delta u +m^2 u +W(u,\nabla u)=0
\end{equation}
where $u=(u^1,\ldots,u^q)$ denotes a function $u:M \longrightarrow\R^q$ and $F$ an analytic function $W=(W^1,\ldots,W^q):\left(\R\times\R^{n+1}\right)^q\longrightarrow\R^q$. Then we want to define an analogue of \refeq{magic} for \refeq{kg.system}. 

In the heuristic presentation of Section \ref{formalsection} it suffices to replace the formal algebra $\mathcal{A}$ by $ \mathcal{A}_q$ spanned over $\R$ by the symbols
$(\source_i(x),\but_j(y))_{x,y\in M;\ i,j\in\lent 1,q\rent}$
which satisfy properties analogous to (i), (ii) and (iii). More precisely we suppose that these symbols satisfy
\begin{gather}
    \forall x,y\in M,\quad \forall i,j\in\lent 1,q\rent;\quad [\source_i(y),\but_j(x)] = \delta_{ij} G(y-x),\label{Hsystem1}\\
    \forall x,y\in M,\quad \forall i,i'\in\lent 1,q\rent;\quad [\source_i(y),\source_{i'}(x)] = [\but_i(y),\but_{i'}(x)] = 0,\label{Hsystem2}
\end{gather}
and $\source_i(x)$, $\but_j(x)$ depend smoothly on $x$ and $\square \source_i +
  m^2 \source_i = 0$ and $\square \source_j + m^2 \source_j = 0$. Moreover we suppose the existence of two representations $\F$ and $\F^*$ and two elements $|0\rangle$ and $\langle 0|$ belonging respectively to $\F$ and $\F^*$ such that
$\forall y\in M$, $\forall j\in\lent 1,q\rent$, $\forall x\in M$, $\forall i\in\lent 1,q\rent$,
\[
\phi^\td_j(y)|0\rangle = 0\quad;\quad \langle 0|\phi^\tx_i(x) = 0 \quad;\quad \langle0|0\rangle = 1.
\]
So in the case where $q=1$ we just recover the algebra ${\cal A}$ considered previously. Then, given an element  $|f\rangle$ of $\F$ we consider the quantity ${\cal F}_t^\varphi[u]_t$ equal to
\begin{multline}\label{magic.system}
    \langle 0|\hbox{exp}\left(\int_{y^0=t}\left\{u(y){\gd\over \partial y^0}\bigg|\phi^\td(y)\right\} d\vec{y} \right) T\hbox{exp}\left(-\int_{0<z^0<t}\lambda \left\{F(\phi^\tx(z))\big|\phi^\td(z)\right\} dz\right) |f\rangle,
\end{multline}
where
\begin{align*}
        \left\{u(y){\gd\over \partial y^0}\bigg|\phi^\td(y)\right\}&:= \sum_{l=1}^q u^l(y){\gd\over \partial y^0}\phi^\td_l(y)\\
        \left\{F(\phi^\tx(z))\big|\phi^\td(z)\right\}&:= \sum_{l=1}^q F^l(\phi^\tx(z))\phi^\td_l(z).
\end{align*}
Then \refeq{magic.system} is the analogue of \refeq{magic} for system \refeq{kg.system}. Indeed using the formal computations we have done for the case $q=1$, it is very easy to see that formally ${\cal F}_t^\varphi[u]_t$ given by \refeq{magic.system} does not depend on $t$ if $u$ satisfies \refeq{kg.system}. 

In order to give a rigorous meaning to \refeq{magic.system}, we let, for $s>n/2$, $H^s(\R^n,\R^q)$ to be the space of vector valued functions $f:\R^n\longrightarrow \R^q$ such that $(m^2+|\xi|^2)^{s/2}|\chapo{f}(\xi)|$ belongs to $L^2(\R^n)$. Then we set
\[
    \mathcal{E}_0^{s+1}:= \{\varphi\in \mathcal{C}^0(\R,H^{s+1}(\R^n,\R^q))\cap \mathcal{C}^1(\R,H^s(\R^n,\R^q))|\ \square \varphi + m^2\varphi = 0\}.
\]
The definition of the `Fock space' $\F_r$ follows then the same line as in Section \ref{sectionsetting}. The creation operators $\phi^\tx_j(x)$ acting on $\F_r$, for $x\in M$ and $j\in\lent 1,q\rent$ are just defined as:
\[
\phi^\tx_j(x):|f\rangle\in\F_R\longmapsto |x\rangle^j |f\rangle \in\F_R
\]
where $|x\rangle^j$ denote the functional $|x\rangle^j:\varphi\in {\cal E}_0^{s+1}\longmapsto \varphi^j(x)$ which is well defined since $s>n/2$ (here $\varphi^1$, \dots, $\varphi^q$ denote the components of $\varphi$). The idea for defining the annihiliation operators $\phi^\td_j(y)$ is to set, for $|f\rangle = \sum_{p\ge 0}|f_p\rangle\in \F_R$,
\[
    \forall \varphi\in {\cal E}_0^{s+1};\ \left(\phi^\td_j(y)\right)(\varphi):= \sum_{p=1}^\infty p|f_p\rangle\left((\Gamma_y e_j) \otimes \varphi\otimes \cdots\otimes \varphi\right)
\]
where $e_j$ denotes the $j$--th vector of the canonical basis of $\R^q$. Again this definition is inconsistant since we know that $\Gamma_y$ does not belong to ${\cal E}_0^{s+1}$ if $s>n/2$. However by using the same constructions as in Section \ref{sectionsetting} there are no difficulties to define the smeared versions of these operators:
\[
    \begin{array}{lcr}
    \displaystyle \int_{y^0=0} f(y) \but_j(y) dy & \text{ or } &
    \displaystyle \int_{0<z^0<t} W^j(\phi^\tx(z))\phi^\td_j(z) dz.
    \end{array}
\]
One can adapt the previous results in order to make sense of \refeq{magic.system} and to extend Theorem \ref{mainthm} to this situation.\\

\noindent An example of application is the wave equation for
maps $u$ from $\R^{n+1}$ to the sphere $S^q\subset \R^{q+1}$. These are the maps into $\R^{q+1}$ which satisfy the pointwise contraint $|u|^2= 1$ and the wave map equation
\[
\square u + \left(|{\partial u\over \partial t}|^2 - |\vec{\nabla}u|^ 2\right) u = 0,
\]
and our result applies to it with $W(u,\nabla u) =  \left(|{\partial u\over \partial t}|^2 - |\vec{\nabla}u|^ 2\right) u - m^2u$.

\appendix

\section{Appendix}
\subsection{Some results about the space $H^s(\R^n)$, for $s>n/2$}
\begin{lemm}\label{lemmapp1} Assume that $s>n/2$. Then for any function $\varphi\in H^s(\R^n)$,
\begin{equation}\label{Linfinity}
||\varphi||_{L^\infty} \leq {1\over
\sqrt{2\pi}^n}||\widehat{\varphi}||_{L^1}\leq {1\over
\sqrt{2\pi}^n}I(n,m,s)||\varphi||_{H^s},
\end{equation}
where $\widehat{\varphi}$ is the Fourier transform of $\varphi$ and
\[
I(n,m,s):= \left(\int_{\R^n} \epsilon(\xi)^{-2s}d\xi\right)^{1/2} = m^{{n\over 2}-s}\sqrt{|S^{n-1}|} \left(\int_0^\infty {t^{n-1}dt\over (1+t^2)^s}\right)^{1/2}.
\]
Moreover $H^s(\R^n)\subset \mathcal{C}^0(\R^n)$.
\end{lemm}
Note that $I(n,m,s)\leq m^{{n\over 2}-s}\sqrt{|S^{n-1}|} \sqrt{2s\over n(2s-n)}$. Moreover $|S^{n-1}| = 2{\pi^{n/2}\over \Gamma(n/2)}$.\\
\emph{Proof} --- The inequality $||\varphi||_{L^\infty}\leq {1\over
\sqrt{2\pi}^n}||\widehat{\varphi}||_{L^1}$ is straightforward, hence
it suffices to estimate $||\widehat{\varphi}||_{L^1}$ in terms of
$||\varphi||_{H^s}$. This follows by using the Cauchy--Schwarz
inequality
\[
||\widehat{\varphi}||_{L^1} = \int_{\R^n}\left(\epsilon(\xi)^s|\widehat{\varphi}(\xi)|\right) {d\xi \over \epsilon(\xi)^s} \leq ||\epsilon^s\widehat{\varphi}||_{L^2}\sqrt{\int_{\R^n} \epsilon(\xi)^{-2s}d\xi} = ||\varphi||_{H^s} I(n,m,s).
\]
But by using spherical coordinates on $\R^n$:
\[
I(n,m,s)^2 = |S^{n-1}| \int_0^\infty {r^{n-1}dr\over (m^2+r^2)^s} = |S^{n-1}|m^{n-2s}\int_0^\infty {t^{n-1}dt\over (1+t^2)^s},
\]
where we have posed $t= r/m$. So (\ref{Linfinity}) follows. Then one can deduce that all functions in $H^s(\R^n)$ are continuous by using the density of smooth maps in $H^s(\R^n)$ and (\ref{Linfinity}).

\begin{lemm}\label{lemmapp2} Assume that $s>n/2$. Then for any functions $\varphi,\psi\in H^s(\R^n)$ the product $\varphi\psi$ belongs to $H^s(\R^n)$ and
\begin{equation}\label{productHs}
||\varphi\psi||_{H^s} \leq 2^s I(n,m,s) ||\varphi||_{H^s}||\psi||_{H^s}.
\end{equation}
Hence $\left(H^s(\R^n),+,\cdot\right)$ is an algebra.\bbox
\end{lemm}
\emph{Proof} --- Since $\widehat{\varphi\psi} = \widehat{\varphi}* \widehat{\psi}$, we have
\begin{equation}\label{app1}
|\epsilon(\xi)^s\widehat{\varphi\psi}(\xi)| = \epsilon(\xi)^s
\left| \int_{\R^n}\widehat{\varphi}(\xi-\eta) \widehat{\psi}(\eta)  d\eta\right| \leq \epsilon(\xi)^s\int_{\R^n} |\widehat{\varphi}(\xi-\eta)|\ |\widehat{\psi}(\eta)| d\eta.
\end{equation}
But by the triangular inequality:
\[
\epsilon(\xi) = \sqrt{(m+0)^2 + |(\xi-\eta) + \eta|^2} \leq  \sqrt{m^2 + |\xi-\eta|^2} + \sqrt{0^2 + |\eta|^2}
\leq \epsilon(\xi-\eta) + \epsilon(\eta)
\]
and hence
\[
\epsilon(\xi)^s \leq 2^{s-1}\left( \epsilon(\xi-\eta)^s + \epsilon(\eta)^s\right).
\]
By using this in (\ref{app1}) we obtain
\[
\begin{array}{ccl}
|\epsilon(\xi)^s\widehat{\varphi\psi}(\xi)| & \leq & 2^{s-1}\int_{\R^n} |\widehat{\varphi}(\xi-\eta)|\ |\widehat{\psi}(\eta)| \left( \epsilon(\xi-\eta)^s + \epsilon(\eta)^s\right) d\eta\\
& = & 2^{s-1}\left(\int_{\R^n} \epsilon(\xi-\eta)^s|\widehat{\varphi}(\xi-\eta)|\ |\widehat{\psi}(\eta)| + \int_{\R^n} \epsilon(\eta)^s|\widehat{\varphi}(\xi-\eta)|\ |\widehat{\psi}(\eta)|\right)\\
& = & 2^{s-1}\left((\epsilon^s|\widehat{\varphi}|)*|\widehat{\psi}|(\xi) + |\widehat{\varphi}| * (\epsilon^s|\widehat{\psi}|)(\xi) \right).
\end{array}
\]
We now deduce by using Young's inequality $||f*g||_{L^2}\leq ||f||_{L^2}||g||_{L^1}$ that
\[
\begin{array}{ccl}
||\varphi\psi||_{H^s} & = & ||\epsilon^s\widehat{\varphi\psi}||_{L^2}\\
& \leq & 2^{s-1}\left(||(\epsilon^s|\widehat{\varphi}|)*|\widehat{\psi}|\,{||}_{L^2} + {||}\,|\widehat{\varphi}| * (\epsilon^s|\widehat{\psi}|)||_{L^2}\right) \\
& \leq & 2^{s-1}\left(||\epsilon^s\widehat{\varphi}||_{L^2}||\widehat{\psi}||_{L^1} + ||\widehat{\varphi}||_{L^1} ||\epsilon^s\widehat{\psi}||_{L^2}\right)\\
& = & 2^{s-1}\left(||\varphi||_{H^s}||\widehat{\psi}||_{L^1} + ||\widehat{\varphi}||_{L^1} ||\psi||_{H^s}\right).
\end{array}
\]
But because of Lemma \ref{lemmapp1} we deduce that
\[
||\varphi\psi||_{H^s} \leq 2^{s-1}\left(||\varphi||_{H^s} I(n,m,s)||\psi||_{H^s} + I(n,m,s)||\varphi||_{H^s}||\psi||_{H^s}\right) = 2^s I(n,m,s)||\varphi||_{H^s}||\psi||_{H^s}.
\]
Hence Lemma \ref{lemmapp2} is proved. \bbox


\subsection{About the operator $v\longmapsto v\sharp_tG^{(k)}$}\label{abouttheop}
For any $r\in \R$, $k\in \N$, $t\in \R$ and for any function $v\in H^r(\R^n)$ the definition of $v\sharp_tG^{(k)}$ given by (\ref{vsharpG}) is equivalent to: $\forall (x^0,\vec{x})\in \R^{n+1}$, $v\sharp_tG^{(k)}(x^0,\vec{x}) = v(\vec{x})*g^{(k)}_{t-x^0}$, where $g^{(k)}_t(\vec{x}):= {\partial^kG\over \partial t^k} (t, \vec{x})$. This is indeed a consequence of $g^{(k)}_t(-\vec{x}) = g^{(k)}_t(\vec{x})$. Alternatively, since the spatial Fourier transform of ${\partial^kG\over \partial t^k}$ is
\[
{\partial^k\widehat{G}\over \partial t^k}(t,\xi) = {1\over \sqrt{2\pi}^n} \hbox{Re}\left(i^{k-1}\epsilon(\xi)^{k-1} e^{i\epsilon(\xi)t}\right),
\]
$v\sharp_tG^{(k)}$ can be defined through its spatial Fourier transform $\widehat{v*g^{(k)}_{t-x^0}} = \sqrt{2\pi}^n\widehat{\widehat{v}g^{(k)}_{t-x^0}}$:
\begin{equation}\label{app2}
\widehat{v\sharp_tG^{(k)}}(x^0,\xi) = \hbox{Re}\left(i^{k-1}\epsilon(\xi)^{k-1} e^{i\epsilon(\xi)(t-x^0)}\right) \widehat{v}(\xi).
\end{equation}
\begin{lemm}\label{lemmapp3}
Let $r\in \R$, $k,\ell\in \N$ and $v\in H^r(\R^n)$. Then
\begin{itemize}
\item[(i)] $\forall t,x^0\in \R$, $v\sharp_tG^{(k)}(x^0,\cdot)\in H^{r-k+1}(\R^n)$ and $||v\sharp_tG^{(k)}(x^0,\cdot)||_{H^{r-k+1}}\leq ||v||_{H^r}$;
\item[(ii)] $\forall t\in \R$, $v\sharp_tG^{(k)}\in \mathcal{C}^{\ell}(\R,H^{r-\ell-k+1}(\R^n))$ and, if $0\leq j\leq \ell$,
\[
{\partial^j\over \partial (x^0)^j}\left(v\sharp_tG^{(k)}\right) = (-1)^j v\sharp_tG^{(k+j)};
\]
\item[(iii)] $v\sharp_tG^{(k)}$ is a weak solution of (\ref{kg}), i.e. $v\sharp_tG^{(k)}\in \mathcal{E}^{r-k+1}_0$ and
\[
 ||v\sharp_tG^{(k)}||_{\mathcal{E}^{r-k+1}} = ||v||_{H^r}.
\]
\item[(iv)] The map $t\longmapsto v\sharp_t G^{(k)}$ belongs to ${\cal C}^\ell(\R,{\cal E}^{r-l-k+1})$ and for all $j\in\lent 0,\ell\rent$
\[
    {\partial^j\over \partial t^j}\left(v\sharp_tG^{(k)}\right) = v\sharp_tG^{(k+j)}
\]
\end{itemize}
\end{lemm}
\emph{Proof} --- All these properties are consequences of (\ref{app2}) and of the obvious observations that, if we denote $H_k(t-x^0,\xi):=  \hbox{Re}\left(i^{k-1}\epsilon(\xi)^{k-1} e^{i\epsilon(\xi)(t-x^0)}\right)$ then $H_k$ is smooth and $|H_k(t-x^0,\xi)| \leq \epsilon(\xi)^{k-1}$. The proof of (ii) and (iv) requires furthermore the use of Lebesgue's theorem. The proof of (iii) follows from ${\partial^2\over \partial t^2}\left(\widehat{v\sharp_tG^{(k)}}\right) = - \epsilon^2\widehat{v\sharp_tG^{(k)}}$. \bbox

\end{document}